\documentclass{amsart} 
\usepackage{latexsym}
\usepackage{amsmath}
\usepackage{amssymb}
\usepackage{amsfonts} 
\usepackage{eucal} 
\usepackage{cmmib57} 
\usepackage{amsbsy}
\usepackage{amsthm}
\usepackage{amscd}
\usepackage[dvips]{graphics}
\usepackage[all]{xy}
\usepackage{color}
\usepackage{hyperref}

\newtheorem{thm}{Theorem}[section]
\newtheorem*{thm*}{Theorem}
\newtheorem{lemma}[thm]{Lemma}
\newtheorem{prop}[thm]{Proposition}
\newtheorem{cor}[thm]{Corollary}
\newtheorem*{cor*}{Corollary}

\theoremstyle{definition}
\newtheorem{defn}[thm]{Definition}
\newtheorem{example}[thm]{Example}

\theoremstyle{remark}
\newtheorem{remark}[thm]{Remark}

\newcommand {\real}  {\ensuremath{\mathbb{R}}}
\newcommand {\intg}  {\ensuremath{\mathbb{Z}}}

\newcommand {\cplx}  {\ensuremath{\mathbb{C}}}

\newcommand {\Hom}   {\ensuremath{\operatorname{Hom}}}

\newcommand {\cok}   {\operatorname{coker}}
\newcommand {\im}    {\operatorname{im}}

\newcommand {\codim} {\ensuremath{\operatorname{codim}}}

\newcommand {\ip}    {\ensuremath{I^{\overline{p}}}}
\newcommand {\iq}    {\ensuremath{I^{\overline{q}}}}
\newcommand {\imi}   {\ensuremath{I^{\overline{m}}}}

\newcommand {\redh}  {\ensuremath{\widetilde{H}}}

\newcommand {\loc}   {\ensuremath{\operatorname{loc}}}

\newcommand {\cone}  {\ensuremath{\operatorname{cone}}}

\newcommand {\pt}    {\ensuremath{\operatorname{pt}}}
\newcommand {\id}    {\ensuremath{\operatorname{id}}}

\newcommand {\ft}   {\ensuremath{\operatorname{ft}}}
\newcommand {\hft}   {\ensuremath{\widehat{\operatorname{ft}}}}

\newcommand {\dR}  {\ensuremath{\operatorname{dR}}}
\newcommand {\inc}   {\ensuremath{\operatorname{inc}}}

\newcommand{\calH}{\mathcal{H}}
\newcommand {\dS}{\ensuremath{\tilde{d}_\Sigma}}
\newcommand {\dL}{\ensuremath{\tilde{d}_L}}

\newcommand{\delL}{\tilde{\delta}_L}
\newcommand{\del}{\partial}
\newcommand{\olM}{\overline{M}}
\newcommand{\olN}{\overline{N}}
\newcommand{\starL}{\tilde{*}_L}
\newcommand{\olm}{{\overline{m}}}
\newcommand{\oln}{{\overline{n}}}
\newcommand{\olp}{{\overline{p}}}
\newcommand{\olq}{{\overline{q}}}
\newcommand{\onto}{\twoheadrightarrow}
\newcommand{\R}{\mathbb{R}}

\begin{document}


\title{Hodge Theory For Intersection Space Cohomology}

\author{Markus Banagl}

\address{Mathematisches Institut, Universit\"at Heidelberg,
  Im Neuenheimer Feld 288, 69120 Heidelberg, Germany}

\email{banagl@mathi.uni-heidelberg.de}

\thanks{The first author was in part supported by a research grant of the
 Deutsche Forschungsgemeinschaft.}

\date{February, 2015}

\author{Eug\'enie Hunsicker}

\address{Department of Mathematical Sciences, Loughborough University, Loughborough LE11 3TU, UK}

\email{E.Hunsicker@lboro.ac.uk}




\begin{abstract}
Given a perversity function in the sense of intersection homology theory,
the method of intersection spaces assigns to certain oriented stratified spaces
cell complexes whose ordinary reduced homology with real coefficients satisfies
Poincar\'e duality across complementary perversities.
The resulting homology theory is well-known not to be isomorphic to
intersection homology. For a two-strata pseudomanifold with product link bundle, 
we give a description of the cohomology of
intersection spaces as a space of weighted $L^2$ harmonic forms on the regular part,
equipped with a fibred scattering metric. Some consequences of our methods for the signature
are discussed as well.
\end{abstract}

\maketitle


\tableofcontents


\section{Introduction}

Classical approaches to Poincar\'e duality on singular spaces are Cheeger's
$L^2$ cohomology with respect to suitable conical metrics on the regular part of the space
(\cite{cheeger1}, \cite{cheeger2}, \cite{cheeger3}),
and Goresky-MacPherson's intersection homology, depending on a perversity parameter.
Cheeger's Hodge theorem asserts that the space of $L^2$ harmonic forms on the regular
part is isomorphic to the linear dual of intersection homology for the middle perversity, at least
when $X$ has only strata of even codimension, or more generally, is a so-called Witt space.

More recently, the first author has introduced and investigated a different, spatial perspective
on Poincar\'e duality for singular spaces (\cite{Ba-HI}).
This approach associates to certain classes of singular spaces $X$ a cell complex
$\ip X$, which depends on a perversity $\bar{p}$ and is called an \emph{intersection space}
of $X$. Intersection spaces are required to be generalized geometric Poincar\'e complexes in the
sense that when $X$ is closed and oriented, there is a Poincar\'e duality isomorphism
$\redh^{i} (\ip X;\real) \cong \redh_{n-i} (\iq X;\real)$, where
$n$ is the dimension of $X$, $\olp$ and $\olq$ are complementary perversities in the sense
of intersection homology theory, and $\redh^*, \redh_*$ denote reduced singular (or cellular)
cohomology and homology, respectively. The present paper is concerned with $X$ that have two strata such that
the bottom stratum has a trivializable link bundle. The construction of intersection spaces
for such $X$, first given in Chapter 2.9 of \cite{Ba-HI},
is described here in more detail in Section \ref{sec.intsp}. The fundamental principle,
even for more general $X$, is to replace links by their Moore approximations, a concept 
from homotopy theory Eckmann-Hilton dual to the concept of Postnikov approximations.
The resulting (co)homology theory 
$HI^*_{\olp} (X) = H^* (\ip X;\real),$ $HI_*^{\olp} (X) = H_* (\ip X;\real)$
is \emph{not} isomorphic to intersection (co)homology 
$IH^*_{\olp} (X;\real), IH_*^{\olp} (X;\real)$. The theory $HI^*$ has had applications in
fiber bundle theory and computation of equivariant cohomology (\cite{Ba-equiv}),
K-theory (\cite[Chapter 2.8]{Ba-HI}, \cite{spiegel}),
algebraic geometry
(smooth deformation of singular varieties (\cite{bmjta}, \cite{bmjs}), perverse sheaves \cite{bbm},
mirror symmetry \cite[Chapter 3.8]{Ba-HI}),
and theoretical Physics (\cite[Chapter 3]{Ba-HI}, \cite{bbm}).
Note for example that the approach of intersection spaces makes it straightforward
to define intersection $K$-groups by $K^* (\ip X)$.
These techniques are not accessible to classical intersection cohomology.
There are also applications to Sullivan formality of singular spaces:
Given a perversity $\bar{p}$, call a pseudomanifold $X$ \emph{$\bar{p}$-intersection formal}
if $\ip X$ is formal in the usual sense. Then the results of \cite{bmjta}
show that under a mild torsion-freeness hypothesis on the homology of links, 
complex projective hypersurfaces $X$ with only isolated singularities,
whose monodromy operators on the cohomology of the associated Milnor fibers are trivial, 
are 
middle-perversity ($\bar{m}$) intersection formal, since there is an algebra isomorphism from
$HI^*_{\bar{m}} (X)$ to the ordinary cohomology algebra of a nearby smooth deformation, which
is formal, being a K\"ahler manifold. This agrees nicely with the result of
\cite[Section 3.4]{cst}, where it is shown that any nodal hypersurface in $\cplx \mathbb{P}^4$
is ``totally'' (i.e. with respect to an algebra that involves all perversities at once) intersection formal.\\

A de Rham description of $HI^*_{\olp} (X)$ has been given in \cite{Ba-new}
for two-strata spaces whose link bundle is flat with respect to the isometry group 
of the link. Under this assumption, a subcomplex $\Omega I_\olp^* (M)$
of the complex $\Omega^* (M)$ of all smooth differential forms on the
top stratum $M=X-\Sigma,$ where $\Sigma \subset X$ is the singular set, has been defined
such that for isolated singularities there is a de Rham isomorphism
$HI^*_{\dR,\olp} (X) \cong \redh^* (\ip X;\real)$, where
$HI^j_{\dR, \olp} (X) = H^j (\Omega I_\olp^*(M))$. This result has been
generalized by Timo Essig to two-strata spaces with product link bundle
in \cite{Ess}. In \cite{Ba-new} we prove furthermore that wedge product followed by integration
over $M$ induces a nondegenerate intersection pairing
$\cap_{HI}: HI^j_{\dR,\olp} (X) \otimes HI^{n-j}_{\dR,\olq} (X) \to \real$
for complementary $\olp$ and $\olq$. The construction of $\Omega I_\olp^* (M)$
for the case of a product link bundle (i.e. the case relevant to this paper)
is reviewed here in detail in Section \ref{ssec.derhamcplxhi}. \\

In the present paper, we find for every perversity $\bar{p}$ 
a Hodge theoretic description of the theory $HI^*_{\olp} (X)$;
that is, we find a Riemannian metric on $M$ (which is very different from Cheeger's class of
metrics) and a suitable space of $L^2$ harmonic forms with respect to this metric
(the \emph{extended weighted} $L^2$ harmonic forms for suitable weights), such that
the latter space is isomorphic to $\redh I^*_{\olp} (X)\cong HI^j_{\dR, \olp} (X)$.
Assume for simplicity that $\Sigma$ is connected.
If $L$ denotes the link of $\Sigma$ in $X$ and $\olM$ is the compact manifold with
boundary $\partial \olM = L\times \Sigma$ and interior $M$ (called the ``blowup'' of $X$),
then a metric $g_{fs}$ on $M$ is called a product type fibred scattering metric if near
$\partial \olM$ it has the form
\[ g_{fs} = \frac{dx^2}{x^4} + g_\Sigma + \frac{g_L}{x^2}, \]
where $g_L$ is a metric on the link and $g_\Sigma$ a metric on the singular set,
see Section \ref{sec.notation}. If $\Sigma$ is a point, then $g_{fs}$ is a scattering metric.

Given a weight $c$, weighted $L^2$ spaces $x^c L^2_{g_{fs}}\Omega_{g_{fs}}^*(M)$
are defined in Section \ref{sec.hodgehi}.
The space $\mathcal{H}_{ext}^* (M, g_{fs}, c)$ of extended weighted $L^2$ harmonic
forms on $M$ consists of all those forms $\omega$ which are in the kernel of
$d+\delta$ (where $\delta$ is the formal adjoint of the exterior derivative $d$ and depends
on $g_{fs}$ and $c$) and in $x^{c-\epsilon} L^2_{g_{fs}}\Omega_{g_{fs}}^*(M)$
for every $\epsilon >0$, cf. Definition \ref{def.extforms}.
Extended $L^2$ harmonic forms are already present in Chapter 6.4 of Melrose's 
monograph \cite{meaps}.
Then our Hodge theorem is:

\begin{thm}\label{hodge}
Let $X$ be a (Thom-Mather) stratified pseudomanifold with smooth, connected singular stratum $\Sigma \subset X$.  Assume that
the link bundle $Y \to \Sigma$ is a product $L \times \Sigma \to \Sigma$, where $L$ is a smooth manifold
of dimension $l$.
Let $g_{fs}$ be an associated product type fibred scattering metric on $M = X - \Sigma$.  Then 
\[
 HI^*_{\dR, \olp} (X)\cong \mathcal{H}_{ext}^*\left(M, g_{fs}, \frac{l-1}{2}-\olp(l+1)\right) .
\]
\end{thm}
As a corollary, the spaces
$\mathcal{H}_{ext}^* \left(M, g_{fs}, \frac{l-1}{2}-\olp(l+1)\right)$ satisfy
Poincar\'e duality across complementary perversities.
Using an appropriate Hodge star operator, this has been shown directly by
the second author in \cite{Hu3}.
It is worth noting that on the space of 
extended harmonic forms, integration does not give a 
well-defined intersection pairing on the right hand side.  Thus it would be interesting in the future to consider 
how to realise the intersection pairing on extended harmonic forms.

The strategy of the proof of Theorem \ref{hodge} is as follows: First we relate $HI^*_{\dR,\olp} (X)$  and
$\redh_* (\ip X;\real)$ to intersection cohomology and intersection homology, respectively.
To do this, we introduce in Section \ref{sec.notation} the device of a 
\emph{conifold transition} $CT(X)$ associated to an $X$ as in the Hodge theorem.
The conifold transition arose originally in theoretical Physics and algebraic geometry as
a means of connecting different Calabi-Yau $3$-folds to each other by a process of
deformations and small resolutions, see Chapter $3$ of \cite{Ba-HI} for more information.
Topologically, such a process also arises in manifold surgery theory when $\Sigma$ is an 
embedded sphere with trivial normal bundle. The relation of $HI$ to $IH$ is then given by the
following theorem.
\begin{thm}\label{ihhihom} (Homological version.)
Let $X$ be an $n$-dimensional stratified pseudomanifold with smooth 
nonempty singular stratum $\Sigma \subset X$.  
Assume that $\Sigma$ is closed as a manifold and
the link bundle $Y \to \Sigma$ is a product bundle $L \times \Sigma \to \Sigma$, where the link $L$ is a smooth closed manifold
of dimension $l$.  Then the reduced homology $\redh I^\olp_* (X)$ 
of the intersection space of perversity $\olp$ 
is related to the intersection homology of the conifold transition $CT(X)$ of $X$ by:
\[
\redh I^\olp_j(X) \cong IG^{(n-1-\olp(l+1)-j)}_j(CT(X)), 
\]
where for a pseudomanifold $W$ with one singular stratum of codimension $c$,
\[
IG_j^{(k)}(W) = \frac{IH_j^\olq(W) \oplus IH_j^{\olq'}(W)}{ {\rm Im}(IH_j^\olq(W) \to IH_j^{\olq'}(W))},
\]
with $\olq(c) = k-1$ and $\olq'(c) = k$.
\end{thm}
Note that when both $\olp (l+1)$ and the degree $j$ are large, one must allow negative values for 
$k$ in the quotient $IG^{(k)}_j$. Therefore, the perversity functions $\olq$ considered in this paper
are not required to satisfy the Goresky-MacPherson conditions, but are simply arbitrary integer valued functions.
This, in turn, necessitates a minor modification in the definition of intersection homology.
The precise definition of $IH_*^{\olq}$ used in the above theorem is provided in Section \ref{sec.ih}
and has been introduced independently by Saralegi \cite{saralegi} and by Friedman \cite{friedmanbook}.
In the following de Rham version of the above result, the link and the singular stratum are assumed to be
orientable, as our methods rely on the availability of the Hodge star operator.
\begin{thm}\label{ihhi} (De Rham Cohomological version)
Let $X$ be a stratified pseudomanifold with smooth singular stratum $\Sigma \subset X$.  Assume that
$\Sigma$ is closed and orientable, and
the link bundle $Y \to \Sigma$ is a product bundle $L \times \Sigma \to \Sigma$, where the link $L$ is a smooth closed 
orientable manifold
of dimension $l$.  Then the de Rham cohomology $HI_{\dR, \olp}^*(X)$ can be described in terms of intersection cohomology by:
\[
HI^j_{\dR, \olp} (X) \cong IG_{(j+1-k)}^j(CT(X)), 
\]
where $k = l-\olp(l+1)$ and for a pseudomanifold $W$ with one singular stratum,
\[
IG_{(q)}^j(W) = \frac{IH_{(q-1)}^j(W) \oplus IH_{(q)}^j(W)}{ {\rm Im}(IH_{(q-1)}^j(W) \to IH_{(q)}^j(W))},
\]
with the notation $IH_{(q)}^*(W)$ as given in Equation (\ref{ihpoincare}).
\end{thm}
We do \emph{not} deduce the cohomological version from the homological one by universal coefficient theorems,
but prefer to give independent proofs for each version. The proof of the homological version
uses Mayer-Vietoris techniques while the proof of the cohomological version compares differential forms
in the various de Rham complexes on $M$. (The regular part of the conifold transition coincides with
the regular part $M$ of $X$.)
Finally, we appeal to a result (Theorem \ref{ihthm} in the present paper) of the second author 
(\cite{Hu3}), which relates extended weighted
$L^2$ harmonic forms with respect to a fibred cusp metric $g_{fc}$
to the $IG^j_{(q)}$ arising in Theorem \ref{ihhi} above. This leads in a natural way to
fibred scattering metrics because fibered cusp metrics 
\[  g_{fc} = \frac{dx^2}{x^2} + g_L + x^2 g_\Sigma \]
on the conifold transition are conformal to
\[ \frac{1}{x^2} g_{fc} = \frac{dx^2}{x^4} + \frac{1}{x^2} g_L + g_\Sigma, \] 
which is precisely a fibred scattering metric on $X$.\\

For $n$ divisible by $4$, and either $l$ odd or $H^{l/2}(L)=0$ (i.e. $X$ a Witt space),
the nondegenerate intersection pairing 
$\cap_{HI}:HI_{\dR,\olm}^{n/2} (X)\otimes HI_{\dR,\olm}^{n/2} (X)\to \real$
on the middle dimension $n/2$ for the middle perversity
has a signature $\sigma_{HI} (X)$, which in the setting of isolated singularities is equal
to the Goresky-MacPherson signature coming from intersection (co)homology of $X$,
\cite[Theorem 2.28]{Ba-HI}.
  We use Theorem \ref{ihhi} to obtain results about the intersection pairing and signature on $HI^*_{\dR}(X)$.
This turns out to be related to perverse signatures, which are signatures defined for arbitrary perversities on arbitrary
pseudomanifolds from the extended intersection pairing on intersection cohomology.  Perverse signatures are 
defined in the two stratum case in \cite{Hu2} and more generally in \cite{FH}.
\begin{thm}\label{intersection}  
Let $X$ be an $n$-dimensional compact oriented stratified pseudomanifold 
with smooth singular stratum $\Sigma \subset X$.  Assume that
the link bundle $Y \to \Sigma$ is a product bundle $L \times \Sigma \to \Sigma$.
Then the intersection pairing $\cap_{HI}: HI^j_{\dR, \olp} \otimes HI^{n-j}_{\dR, \olq(X)}\to \mathbb{R}$ for dual perversities
$\olp$ and $\olq$ is compatible 
with the intersection pairing on the intersection cohomology spaces $IG^*$ appearing
in Theorem \ref{ihhi}.
When $X$ is an even dimensional Witt space, then the signature $\sigma_{HI}(X)$ of the intersection form on 
$HI^{n/2}_{\dR,\olm} (X)$ is equal both to the signature $\sigma_{IH}(X)$ of the
Goresky-MacPherson intersection form on $IH^{n/2}_{\olm}(X)$, and to 
the perverse signature $\sigma_{IH,\olm}(CT(X)),$ 
that is, the signature of the intersection form on 
\[
{\rm Image}\left(IH^{n/2}_{\olm}(CT(X)) \to IH^{n/2}_\oln(CT(X))\right),
\]
where $\olm$ is the lower middle and $\oln$ the upper middle perversity.
Further, 
\[
\sigma_{HI}(X) =\sigma_{IH} (X) = \sigma_{IH,\olm}(CT(X)) 
=\sigma_{IH}(Z)=\sigma_{HI}(Z) 
=\sigma (\overline{M}),
\] 
where $Z$ is the
one-point compactification of $X-\Sigma$ and $\sigma (\olM)$ is the Novikov signature
of the complement $\overline{M}$ of an open tubular neighborhood of the singular set.
\end{thm}
\begin{remark}
The compactification $Z$ appearing in Theorem \ref{intersection} has one isolated singular point. Since
$X$ is even-dimensional, $Z$ is thus a Witt space and has a well-defined $IH$-signature and a well-defined
$HI$-signature. However, if $X$ satisfies the Witt condition, then $CT(X)$ need not satisfy the Witt condition
and $\sigma_{IH} (CT(X))$ and $\sigma_{HI} (CT(X))$ are a priori not defined.
Therefore, we must use the \emph{perverse} signature $\sigma_{IH,\olm}$ for $CT(X)$ as defined in \cite{Hu2}, \cite{FH}.
\end{remark}
We prove Theorem \ref{intersection} using de Rham theory, Siegel's work
\cite{siegel}, Novikov additivity and results of \cite{bcs}. 
Using different, algebraic methods and building
on results of \cite{Ba-HI}, parts of this theorem were also obtained by Matthias Spiegel in his
dissertation \cite{spiegel}.  \\

\textbf{General Notation.}
Throughout the paper, the following notation will be used.
If $f$ is a continuous map, then $\cone (f)$ denotes its mapping cone.
For a compact topological space $X$, $cX$ denotes the closed cone and
$c^\circ X$ the open cone on $X$.
Only homology and cohomology with real coefficients are used in this paper.
Thus we will write $H_* (X) = H_* (X;\real)$. When $M$ is a smooth manifold,
$H^* (M)$ is generally, unless indicated otherwise, understood to mean de Rham
cohomology. The symbol $\redh_* (X)$ denotes the reduced (singular) homology of $X$;
$\redh^* (X)$ is the reduced cohomology.

\section{The Conifold Transition and Riemannian Metrics}
\label{sec.notation}

Let $X$ be a Thom-Mather stratified pseudomanifold with a single smooth singular stratum, $\Sigma$.  
Let $M = X -\Sigma$.
Assume that the link bundle of $\Sigma$ is a product, $L \times \Sigma$.  
Let $N \subset X$
be an open tubular neighborhood of $X$.  Fix a diffeomorphism
\[
\theta:N -\Sigma \cong L\times \Sigma \times (0,1)
\]
that extends to a homeomorphism 
\[
\tilde{\theta}: N \cong \frac{L \times \Sigma \times [0,1)}{(z,y,0) \sim (z',y,0)} \cong c^\circ (L)\times \Sigma.
\]
Define the blowup 
\[
\overline{M} = 
(X - \Sigma) \cup_\theta (L \times \Sigma \times [0,1) ),
\]
with blowdown map $\beta:\overline{M} \to X$ given away from the boundary by the identity
and near the boundary by the quotient map from $L \times \Sigma \times [0,1)$ to
$(L \times \Sigma \times [0,1))/((z,y,0) \sim (z',y,0))$. The blowup is a smooth manifold with boundary 
$Y=\partial\overline{M} = L \times \Sigma$. Let $\inc:\partial \olM \to \olM$ be the inclusion of the boundary, 
and denote the projections onto the two components by $\pi_L:Y\to L$ and $\pi_\Sigma:Y \to \Sigma$.  
By an abuse of notation, we will also use $\pi_L$ and $\pi_\Sigma$ to denote the projections from 
$N -\Sigma \cong L \times  \Sigma \times (0,1)$ to $L$ and $\Sigma$, respectively.  Let 
$\pi_Y$ denote the projection from $N - \Sigma$ to $Y = L \times \Sigma$.

The \emph{conifold transition} of $X$, denoted $CT(X),$ is defined as 
\[
CT(X) = (X - \Sigma) \cup_\theta (L \times \Sigma \times [0,1))/((z,y,0) \sim (z,y',0)).
\]
The conifold transition is a stratified space with one singular stratum $L$, whose
link is $\Sigma$. However, $CT(X)$ is not always a pseudomanifold:
If $X$ has one isolated singularity $\Sigma = \pt$, then
$CT(X)=\olM$ is a manifold with boundary, the boundary constitutes the bottom stratum
and the link is a point. Since the singular stratum does not have codimension at least two,
this is not a pseudomanifold. If $\Sigma$ is positive dimensional, then $CT(X)$ is
a pseudomanifold.
All of our theorems do apply even when $\dim \Sigma =0$. 
Let $\beta':\overline{M} \to CT(X)$ be the blowdown map for the conifold transition of $X$, given by the quotient map.
Note also the involutive character of this construction, $CT(CT(X))\cong X$.

The coordinate $x$ in $(0,1)$ above may be extended to a smooth boundary defining function on $\overline{M}$, that is, a 
nonnegative function, $x$, whose zero set is exactly $\partial \overline{M}$, and 
whose normal derivative does not vanish at $\partial \overline{M}$.  
We can now define the metrics we will consider on $M$, which may in fact be defined on a broader class of open manifolds.
\begin{defn}\label{cfmetric}
Let $M$ be the interior of a manifold $\olM$ with 
fibration boundary $\partial \olM \cong Y \stackrel{\psi}{\to} \Sigma$ with fibre $L$
and boundary defining function $x$.  
Assume that $Y$ can be covered by bundle charts $U_i \cong V_i \times L$ 
whose transition functions $f_{ij}$ have differentials $df_{ij}$ that are diagonal with 
respect to some splitting $TY \cong TL \oplus H$.
A product type {\em fibred scattering metric} on $M$ is a smooth metric that near $\del \olM$ 
has the form:
\[
g_{fs} = \frac{dx^2}{x^4} + \psi^*ds_\Sigma^2 + \frac{h}{x^2},
\]
where $h$ is positive definite on $TL$ and vanishes on $H$.
\end{defn}
Examples of such metrics are the natural Sasaki metrics \cite{sasaki} on the tangent bundle of a compact manifold, $\Sigma$.
In this case, the boundary fibration of $M = T\Sigma$ is isomorphic to the spherical unit tangent bundle $S^n \to Y \to \Sigma$.  
We note that the condition on $Y$ in this definition is necessary for it to make sense.  If the coordinate
transition functions do not respect the splitting of $TY$, then we cannot meaningfully
scale in just the fibre direction.  This is different from the four types of metrics below,
which can be defined on any manifold with fibration boundary.

A special sub-class of these metrics arises when the boundary fibration is flat with respect
to the structure group ${\rm Isom}(L)$ for some fixed metric $ds^2_L$ on $L$.
In this case, we can require the metric $g_{fs}$ to be 
a product metric
\[
g_{fs} = \frac{dx^2}{x^4} + ds_{V_i}^2 + \frac{1}{x^2} ds_{L}^2
\]
on each chart $(0,\epsilon) \times  U_i \cong (0,\epsilon) \times V_i \times L$
for the given fixed metric on $L$.  In this case, we 
say that $g_{fs}$ is a {\em geometrically flat} fibred scattering metric on $M$.
This flatness condition arises also in the definition of $HI$ cohomology, see \cite{Ba-new}.  This of course
can be arranged when the boundary fibration is a product, as in the case we consider in this paper.

Note that in the case that $\partial\olM$ is a product $L \times \Sigma$, it carries two possible boundary
fibrations:  either $\psi:Y \to \Sigma$ or $\phi: Y \to L$. A fibred scattering metric 
on $M$ associated to the boundary fibration $\psi: \partial\olM\to \Sigma$ is a fibred boundary 
metric on $M$ associated to the dual fibration $\phi: \partial\olM\to L$.  Fibred boundary
metrics on $M$ associated to $\phi$ are conformal to a third class of metrics, called fibred cusp metrics.
These two classes may be defined as follows:

\begin{itemize}
\item 
$g_{fb}$ is called a (product type) \emph{fibred boundary} metric if near $\partial \olM$ it takes
the form
\[
g_{fb} =  \frac{dx^2}{x^4} + \frac{\phi^*ds_L^2}{x^2} + k,
\]
where $k$ is a symmetric two-tensor on $\partial \olM$ which restricts to a
metric on each fiber $\Sigma$ of $\phi:Y \to L$;

\item $g_{fc}$ is called a (product type) \emph{fibred cusp} metric if near $\partial \olM$ it takes the
form
\[
g_{fc} = \frac{dx^2}{x^2} + \phi^*ds_L^2 + x^2 k,
\]
where $k$ is as above.
\end{itemize}

In the case that $\Sigma$ is a point, these two metrics reduce to the well-studied classes of b-metrics
and cusp-metrics, respectively,  see, eg \cite{HHM} for more, and $g_{fs}$ becomes a scattering
metric.

\section{Intersection Spaces}
\label{sec.intsp}

Let $\bar{p}$ be an extended perversity, see Section \ref{sec.ih}.
In \cite{Ba-HI}, the first author introduced a
homotopy-theoretic method that assigns to certain types of
$n$-dimensional stratified pseudomanifolds $X$ CW-complexes
\[ \ip X, \]
the \emph{perversity-$\bar{p}$ intersection spaces} of $X$, such that
for complementary perversities $\bar{p}$ and $\bar{q}$, there is a 
Poincar\'e duality isomorphism
\[ \redh^i (\ip X) \cong \redh_{n-i} (\iq X) \]
when $X$ is compact and oriented, where
$\redh^i (\ip X)$ denotes reduced singular cohomology of $\ip X$ with real coefficients.
If $\bar{p} = \bar{m}$ is the lower middle 
perversity, we will briefly write $IX$ for $\imi X$. The singular cohomology groups
\[ HI^\ast_{\bar{p}} (X) = H^\ast (\ip X), \redh I^\ast_{\bar{p}} (X) = \redh^\ast (\ip X) \]
define a new (unreduced/reduced) cohomology theory for stratified spaces, usually not
isomorphic to intersection cohomology $IH^\ast_{\bar{p}}(X)$.
This is already apparent from the observation that
$HI^\ast_{\bar{p}}(X)$ is an algebra under cup product, whereas
it is well-known that $IH^\ast_{\bar{p}} (X)$ cannot generally,
for every $\bar{p}$, be endowed with a $\bar{p}$-internal algebra
structure. Let us put $HI^\ast (X) = H^\ast (IX).$ 

Roughly speaking, the intersection space $IX$ associated to a singular space $X$ is 
defined by replacing links of singularities by their corresponding Moore approximations, 
i.e. spatial homology truncations. Let $L$ be a simply-connected CW complex, and fix an integer $k$.
\begin{defn}
A \emph{stage-$k$ Moore approximation} of $L$ is 
a CW complex $L_{<k}$ 
together with a structural map $f: L_{<k} \to L$, so that $f_*:H_r (L_{<k}) \to H_r(L)$ 
is an isomorphism if $r<k$, and  $H_r (L_{<k}) \cong 0$ for all $r \geq k$.
\end{defn}
Moore approximations exist for every $k$, see e.g. \cite[Section 1.1]{Ba-HI}.
If $k\leq 0$, then we take $L_{<k} =\varnothing,$ the empty set. 
If $k=1$, we take $L_{<1}$ to be a point.
The simple connectivity assumption is sufficient, but certainly not necessary.
If $L$ is finite dimensional and $k>{\dim} L$, then we take the structural map $f$ to be the identity.
If every cellular $k$-chain is a cycle, then we can choose $L_{<k}=L^{(k-1)},$ 
the $(k-1)$-skeleton of $L$,
with structural map given by the inclusion map,
but in general, $f$ cannot be taken to be the inclusion of a subcomplex. \\

Let $X$ be an $n$-dimensional stratified pseudomanifold 
as in Section \ref{sec.notation}. Assume that the link $L$ of $\Sigma$
is simply connected. Let $l$ be the dimension of $L$.
We shall recall the construction of associated
\emph{perversity $\bar{p}$ intersection spaces}
$\ip X$ only for such $X$, though it is available in more generality.
Set $k=l-\bar{p}(l+1)$ and let $f:L_{<k} \to L$ be a stage-$k$ Moore approximation
to $L$.
Let $\olM$ be the blowup of $X$ with boundary $\partial \olM = Y= L\times \Sigma$.
Let
\[ g: L_{<k} \times \Sigma \longrightarrow M \]
be the composition
\[ L_{<k} \times \Sigma 
   \stackrel{f\times \id_{\Sigma}}{\longrightarrow} 
   L\times \Sigma = \partial \olM \hookrightarrow \olM. \]
The intersection space is the
homotopy cofiber of $g$:

\begin{defn} \label{def.intspnonisol}
The \emph{perversity $\bar{p}$ intersection space}
$\ip X$ of $X$ is defined to be
\[ \ip X = \cone(g) = M \cup_g 
   c(L_{<k} \times \Sigma). \]
\end{defn}
Poincar\'e duality for this construction is Theorem 2.47 of \cite{Ba-HI}.
For a topological space $Z$, let $Z^+$ be the disjoint union of $Z$ with a point.
Recall that the cone on the empty set is a point and hence
$\cone (\varnothing \to Z)=Z^+$.
\begin{prop}  \label{prop.hiextcalc}
Let $\olp$ be an (extended) perversity and let $c$ be the codimension of the 
singular stratum $\Sigma$ in $X$. If $\olp (c)<0$, then
$\redh I^{\olp}_* (X)\cong H_* (\olM,\partial \olM),$ and if 
$\olp (c)\geq c-1$, then $\redh I^{\olp}_* (X)\cong H_* (\olM).$
\end{prop}
\begin{proof}
If $\olp (c)<0,$ then $k>l$ and $L_{<k}=L$ with $f:L_{<k} \to L$ the identity.
It follows that $\ip X = \olM \cup_{\partial \olM} c(\partial \olM)$ and
\[ \redh I^{\olp}_* (X) = \redh_* (\olM \cup_{\partial \olM} c(\partial \olM))
 = H_* (\olM, \partial \olM). \]
If $\olp (c)\geq c-1$, then $k\leq 0,$ so $L_{<k} =\varnothing.$
Consequently, $\ip X = \cone (\varnothing \to \olM) = \olM^+$ and
\[ \redh I^{\olp}_* (X) = \redh_* (\olM^+)=H_* (\olM). \]
\end{proof}

\begin{example}\label{elliptic}\rm Consider the equation
\[ y^2 = x^2 (x-1) \]
or its homogeneous version $v^2 w = u^2 (u-w),$ defining a curve $X$ in $\cplx \mathbb{P}^2$.
The curve has one nodal singularity. Thus it is homeomorphic to a pinched torus, that is, $T^2$ with a meridian collapsed to a point, or, equivalently, a cylinder $I\times S^1$ with coned-off boundary, where $I=[0,1]$.
The ordinary homology group $H_1 (X)$ has rank one, generated by the longitudinal circle (while the meridian circle bounds the cone with vertex at the singular point of $X$).
The intersection homology group $IH_1 (X)$ agrees with the intersection homology of the normalization $S^2$ of $X$ (the longitude in $X$ is not an ``allowed" $1$-cycle, while the meridian bounds an allowed $2$-chain), so:
\[ IH_1 (X) = IH_1 (S^2) = H_1 (S^2)=0. \]
The link of the singular point is $\partial I \times S^1$, two circles. The intersection space $IX$ of $X$ 
is a cylinder $I \times S^1$ together with an interval, whose one endpoint is attached to a point in 
$\{ 0 \} \times S^1$ and whose other endpoint is attached to a point in $\{ 1 \} \times S^1$. 
Thus $IX$ is homotopy equivalent to the figure eight and
\[ H_1 (IX) = \real \oplus \real. \]
\end{example}

\begin{remark}\label{r1}\rm As suggested by the previous example, the middle homology of the 
intersection space $IX$ usually takes into account more cycles than the corresponding intersection homology 
group of $X$. More precisely, for $X^{2k}$ with only isolated singularities $\Sigma$, $IH_k (X)$ is generally 
smaller than both $H_k (X-\Sigma)$ and $H_k (X)$, being a quotient of the former and a subgroup of the latter, while $H_k (IX)$ is generally bigger than both $H_k (X-\Sigma)$ and $H_k (X)$, containing the former as a subgroup and mapping to the latter surjectively, see \cite{Ba-HI}.\end{remark}

One advantage of the intersection space approach is a richer algebraic structure: The 
Goresky-MacPherson intersection cochain complexes $IC_{\olp}^* (X)$
are generally not algebras, unless $\olp$ is the zero-perversity, in
which case $IC_{\olp}^* (X)$ is essentially the ordinary cochain complex
of $X$. (The Goresky-MacPherson intersection product raises perversities
in general.) Similarly, Cheeger's differential complex $\Omega^\ast_{(2)}(X)$
of $L^2$-forms on the top stratum with respect to his conical metric is not an algebra under
wedge product of forms. Using the intersection
space framework, the ordinary cochain complex $C^\ast (\ip X)$ of 
$\ip X$ \emph{is} a DGA, simply by employing the ordinary cup product. 

Another advantage of introducing intersection spaces is the possibility of discussing
the intersection $K$-theory $K^* (\ip X)$, which is not possible using intersection chains,
since nontrivial generalized cohomology theories such as $K$-theory do not factor through
cochain theories.

\section{Intersection Homology}
\label{sec.ih}

Intersection homology groups of a stratified space were introduced by Goresky and MacPherson in 
\cite{gm1}, \cite{gm2}. In order to obtain independence of the stratification, they imposed
on perversity functions $\olp: \{ 2,3,\ldots \} \to \{ 0,1,2,\ldots \}$ the conditions
$\olp (2)=0$ and $\olp (k)\leq \olp (k+1) \leq \olp (k)+1$. Theorems
\ref{ihhihom} and \ref{ihhi}, however, clearly involve perversities that do not satisfy these conditions.
Thus in the present paper, a \emph{perversity} $\olp$ is just a sequence of integers
$(\olp (0), \olp (1), \olp (2),\ldots )$. (This is called an ``extended'' perversity; it is called a ``loose'' perversity in
\cite{king}.) Consequently, we need to use a version of intersection homology that behaves
correctly even for these more general perversities. More precisely, the singular intersection
homology of \cite{king}, which agrees with the Goresky-MacPherson intersection homology,
displays the following anomaly for very large perversity values: If $A$ is a closed $(n-1)$-dimensional
manifold and $c^\circ A$ the open cone on $A$, then the intersection homology of $c^\circ A$
vanishes in degrees greater than or equal to $n-1-\olp (n)$, with one exception: If the degree is $0$ and
$0\geq n-1-\olp (n)$, then the intersection homology of the cone is $\intg$. Now if the perversity $\olp$
satisfies the Goresky-MacPherson growth conditions, then this exception can never arise,
since $\olp (n)\leq n-2$. But if $\olp$ is arbitrary, the exception may very well occur.

To correct this anomaly, we use the modification of Saralegi \cite{saralegi} and, independently,
Friedman \cite{friedmanbook}.
Let $\Delta_i$ denote the standard $i$-simplex and let $\Delta^j_i \subset \Delta_i$ be the
$j$-skeleton of $\Delta_i$.
Let $X$ be any stratified space with singular set $\Sigma$ and $\olp$ be an arbitrary (extended) perversity.
Let $C_* (X) =C_* (X;\real)$ denote the singular chain complex with $\real$-coefficients of $X$.
A singular $i$-simplex $\sigma: \Delta_i \to X$ is called \emph{$\olp$-allowable}
if for every pure stratum $S$ of $X$,
\[ \sigma^{-1}(S) \subset \Delta_i^{i-k+\olp (k)}, \text{ where } k=\codim S. \]
(This definition is due to King \cite{king}.)
For each $i=0,1,2,\ldots,$ let $C^\olp_i (X) \subset C_i (X)$ be the linear subspace generated
by the $\olp$-allowable singular $i$-simplices.
If $\xi \in C^\olp_i (X),$ then its chain boundary $\partial \xi \in C_{i-1} (X)$ can be uniquely written as
$\partial \xi = \beta_\Sigma + \beta,$ where $\beta_\Sigma$ is a linear combination of singular
simplices whose image lies entirely in $\Sigma$, whereas $\beta$ is a linear combination of simplices
each of which touches at least one point of $X-\Sigma$. We set
$\partial' \xi = \beta$ and
\[ IC^{\olp}_i (X) = \{ \xi \in C^\olp_i (X) ~|~ \partial' \xi \in C^{\olp}_{i-1} (X) \}. \]
It is readily verified that $\partial'$ is linear and $(IC^\olp_* (X), \partial')$ is a chain complex. The version of intersection
homology that we shall use in this paper is then given by
\[ IH^\olp_i (X) = H_i (IC^\olp_* (X)). \]
Friedman \cite{friedmanbook} shows that if $\olp (k)\leq k-2$ for all $k$, then $IH^\olp_* (X)$ as defined
here agrees with the definition of singular intersection homology as given by Goresky, MacPherson and King.
If $A$ is a closed $a$-dimensional manifold, then
\begin{equation} \label{equ.ihtrunc}
IH^\olp_i (c^\circ A) \cong \begin{cases}
0,& i\geq a-\olp (a+1)\\ IH^\olp_i (A),& i<a-\olp (a+1),
\end{cases} 
\end{equation}
and this holds even for degree $i=0$, i.e. the above anomaly has been corrected.
If $A$ is unstratified (that is, has only one stratum, the regular stratum), then
$IH^\olp_i (A) = H_i (A)$. If $A$ is however not intrinsically stratified, then one
can in general \emph{not} compute $IH^\olp_* (A)$ by ordinary homology, since
$IH^\olp_*$ is not a topological invariant anymore for arbitrary perversities $\olp$.

According to \cite{friedmanbook}, $IH^\olp_*$ has Mayer-Vietoris sequences:
If $U,V \subset X$ are open such that $X=U\cup V,$ then there is an exact sequence
\begin{equation} \label{equ.ihmvgeneric}
\cdots \longrightarrow IH^\olp_i (U\cap V) \to
  IH^\olp_i (U) \oplus
  IH^\olp_i (V) \to
  IH^\olp_i (X) \to
  IH^\olp_{i-1} (U\cap V) \to \cdots. 
\end{equation}
Furthermore, if $M$ is any (unstratified) manifold (not necessarily compact) and 
$X$ a stratified space, then the K\"unneth formula
\begin{equation} \label{equ.ihkunneth}
IH^\olp_* (M\times X)\cong H_* (M)\otimes IH^{\olp}_* (X) 
\end{equation}
holds if the strata of $M\times X$ are the products of $M$ with the strata of $X$.
(When $\olp (k)\leq k-2$ for all $k$, this is Theorem 4 with $\real$-coefficients in \cite{king}.)
\begin{prop} \label{prop.ihctxextperv}
Let $\olq$ be an (extended) perversity and let $c$ be the codimension of the 
singular stratum $L$ in the conifold transition $CT(X)$. If $\olq (c)<0$, then
$IH^{\olq}_* (CT(X))\cong H_* (\overline{M}),$ and if 
$\olq (c)\geq c-1$, then $IH^{\olq}_* (CT(X))\cong H_* (\overline{M},\partial \overline{M}).$
\end{prop}
\begin{proof}
Set $I_j = IH^{\olq}_j (L\times c^\circ \Sigma)$ and let $s=c-1$ be the dimension of $\Sigma$.
Then by (\ref{equ.ihtrunc}) and (\ref{equ.ihkunneth}),
\[ I_* \cong H_* (L)\otimes IH^{\olq}_* (c^\circ \Sigma) =
  H_* (L) \otimes \tau_{\leq s-1-\olq (s+1)} H_* (\Sigma) =
  H_* (L) \otimes \tau_{\leq c-2-\olq (c)} H_* (\Sigma). \]
If $\olq (c)\geq c-1,$ that is, $c-2-\olq (c)<0$, then $I_* =0$, so 
$IH^{\olq}_* (CT(X))\cong H_* (\overline{M},\partial \overline{M})$ by
the Mayer-Vietoris sequence of the open cover $CT(X)=M \cup (L\times c^\circ \Sigma)$
and the five-lemma.
If $\olq (c)<0,$ that is, $c-2-\olq (c)\geq s$, then $I_* = H_* (L\times \Sigma)$, so 
$IH^{\olq}_* (CT(X))\cong H_* (\overline{M}),$ again by a
Mayer-Vietoris argument.
\end{proof}

\section{Proof of Theorem \ref{ihhihom}}

Let $j$ be any nonnegative degree. 
Throughout this entire section, $j$ will remain fixed and we must establish an isomorphism
$\widetilde{HI}^\olp_j(X) \cong IG^{(n-1-\olp(l+1)-j)}_j(CT(X))$.
Let $\overline{M}$ be the blowup of $X$ and $M$ its interior.
Let $c$ be the codimension of the singular set $L$ in $CT(X)$ and $\widehat{c}$ be the codimension of 
the singular set $\Sigma$ in $X$, that is,
\[ c=n-l,~ \widehat{c} = l+1. \]
Set $k= \widehat{c} -1-\olp (\widehat{c})$ and let $f:L_{<k} \to L$ be a stage-$k$ Moore approximation
of $L$. Then the intersection space $\ip X$ is the mapping cone
$\ip X = \cone (g)$ of the map $g: L_{<k} \times \Sigma \to \overline{M}$ given by the composition
\[ L_{<k} \times \Sigma \stackrel{f\times \id_\Sigma}{\longrightarrow}
 L\times \Sigma = \partial \overline{M} \hookrightarrow \overline{M}.  \]
Let $\gamma: IH_j^\olq (CT(X)) \to IH_j^{\olq'}(CT(X))$ be the canonical map, where
$\olq(c) = n-2-\olp(l+1)-j$ and $\olq'(c) = \olq(c)+1$.
Note that the perversities $\olq, \olq'$ depend on the degree $j$. 
Then by definition
\[ IG_j^{(n-1-\olp(l+1)-j)}(CT(X)) = IH_j^\olq (CT(X)) \oplus \cok (\gamma). \]
The strategy of the proof is to compute intersection homology and $HI$ near the singular stratum
using K\"unneth theorems (``local calculations''), then determine maps (``local maps'') between
these groups near the singularities, and finally to assemble this information to global information
using Mayer-Vietoris techniques. Our arguments do not extend to integer coefficients; field coefficients
are essential. 

We begin with the local calculations.
Let $B_*$ be the homology of the boundary, $B_* = H_* (L\times \Sigma)$,
$T_* = H_* (M)$ the homology of the top stratum, $I_* = IH^\olq_* (L\times c^\circ \Sigma)$,
$J_* = IH^{\olq'}_* (L\times c^\circ \Sigma)$ and
$R_* =H_* (\cone (f\times \id_\Sigma))$. Again, we stress that the graded vector spaces
$I_*$ and $J_*$ depend on the degree $j$.

\begin{lemma} \label{lem.redhconef}
The canonical inclusion $L\to \cone (f)$ induces an isomorphism
\[ \tau_{\geq k} H_* (L)\cong \redh_* (\cone (f)). \]
\end{lemma}
\begin{proof}
The reduced homology of the mapping cone of $f$ fits into an exact sequence
\[ \cdots \to H_i (L_{<k}) \stackrel{f_*}{\longrightarrow} H_i (L) \longrightarrow \redh_i (\cone (f)) 
  \longrightarrow H_{i-1} (L_{<k})
 \stackrel{f_*}{\longrightarrow} H_{i-1} (L) \to \cdots. \]
We distinguish the three cases $i=k$, $i>k$ and $i<k$.
If $i<k$, then $f_*$ on $H_i (L_{<k})$ and on $H_{i-1} (L_{<k})$ is an isomorphism
and thus $\redh_i (\cone f)=0$. If $i=k$, then $H_i (L_{<k})=0$ and $f_*$ on $H_{i-1} (L_{<k})$
is an isomorphism. Therefore, $H_i (L)\to \redh_i (\cone f)$ is an isomorphism.
Finally, if $i>k,$ then both $H_i (L_{<k})$ and $H_{i-1} (L_{<k})$ vanish and again
$H_i (L)\to \redh_i (\cone f)$ is an isomorphism.
\end{proof}
\begin{remark}
If $k\leq 0,$ then $L_{<k} = \varnothing$ and $\cone (f)=L^+$.
It follows that in degree $0$,
\[ (\tau_{\geq k} H_* (L))_0 = H_0 (L)\cong \redh_0 (L^+) = \redh_0 (\cone (f)), \]
in accordance with the Lemma.
\end{remark}

Let $v\in \cone (f)$ be the cone vertex and let $Q= (\cone (f)\times \Sigma)/(\{ v \} \times \Sigma)$,
which is homeomorphic to $\cone (f\times \id_\Sigma)$.
As the inclusion $\{ v \} \times \Sigma \to \cone (f)\times \Sigma$
is a closed cofibration, the quotient map induces an isomorphism
\[ H_* ((\cone (f), \{ v \})\times \Sigma)= 
  H_* (\cone (f)\times \Sigma, \{ v \} \times \Sigma) \cong \redh_* (Q) \cong 
  \redh_* (\cone (f\times \id_\Sigma)). \]
By the K\"unneth theorem for relative homology,
\[ H_* ((\cone (f), \{ v \})\times \Sigma) \cong
  H_* (\cone (f), \{ v \})\otimes H_* (\Sigma) = \redh_* (\cone (f))\otimes H_* (\Sigma). \]
Composing, we obtain an isomorphism
\[ \redh_* (\cone (f\times \id_\Sigma)) \cong \redh_* (\cone (f))\otimes H_* (\Sigma). \]
Composing with the isomorphism of Lemma \ref{lem.redhconef}, we get an isomorphism
\begin{equation} \label{equ.redhconeftstaugeqk}
\redh_* (\cone (f\times \id_\Sigma)) \cong (\tau_{\geq k} H_* (L))\otimes H_* (\Sigma). 
\end{equation}
\begin{remark}
If $k\leq 0,$ then
$L_{<k} \times \Sigma = \varnothing$ and thus
$\cone (f\times \id_\Sigma)=(L\times \Sigma)^+$.
This is consistent with
\[ \cone (f\times \id_\Sigma) \cong Q = \frac{\cone(f) \times \Sigma}{\{ v \} \times \Sigma}
 = \frac{L^+ \times \Sigma}{ \{ v \} \times \Sigma}
 = \frac{\{ v \} \times \Sigma \sqcup L\times \Sigma}{ \{ v \} \times \Sigma}
 = (L\times \Sigma)^+. \]
\end{remark}
It will be convenient to put $a= \olp(l+1)+j-l$; then the relation
\begin{equation} \label{equ.akj} 
a+k=j 
\end{equation}
holds.
We compute the terms $R_*$:
\begin{lemma} \label{lem.computationrj}
If $j>0$, then the isomorphism (\ref{equ.redhconeftstaugeqk}) induces an isomorphism
\[ R_j \cong \bigoplus_{t=0}^a H_{j-t} (L)\otimes H_t (\Sigma). \]
(If $a<0,$ this reads $R_j =0$.)
Furthermore,
\[ R_0 \cong \begin{cases}
\real, & k>0 \\ \real \oplus H_0 (L)\otimes H_0 (\Sigma),& k\leq 0.
\end{cases} \]
\end{lemma}
\begin{proof}
We start by observing that $k$ is independent of $j$.
Now for $j>0$, reduced and unreduced homology coincide, so
\[ R_j =\redh_j (\cone (f\times \id_\Sigma)) \cong  
   \bigoplus_{t=0}^j (\tau_{\geq k} H_* (L))_{j-t} \otimes H_t (\Sigma). \] 
Using (\ref{equ.akj}),
$j-t\geq k$ if and only if $a=j-k\geq t$. Thus
\[ R_j \cong \bigoplus_{t=0}^a H_{j-t} (L) \otimes H_t (\Sigma), \] 
since if $a>j$ and $j<t\leq a,$ then $j-t<0$ so that $H_{j-t} (L)=0$.

In degree $0$, we find
\[ R_0 =\real \oplus \redh_0 (\cone (f\times \id_\Sigma)) 
  \cong \real \oplus (\tau_{\geq k} H_* (L))_0 \otimes H_0 (\Sigma) \]
and $(\tau_{\geq k} H_* (L))_0 =0$ for $k>0$, whereas
$(\tau_{\geq k} H_* (L))_0 = H_0 (L)$ for $k\leq 0$.
\end{proof}

With $s=\dim \Sigma$, we have $s+1 = n-l=c$ and thus according to (\ref{equ.ihtrunc}),
\[ IH^{\olq}_* (c^\circ \Sigma) \cong \tau_{< s-\olq (s+1)} H_* (\Sigma) =
  \tau_{\leq a} H_* (\Sigma), \]
for
\[ s-1-\olq (s+1) = s-1-\olq (c) = n-l-2 -(n-2-\olp (l+1)-j) =a. \]
Consequently by the K\"unneth formula (\ref{equ.ihkunneth}) for intersection homology,
\[
I_j \cong (H_* (L)\otimes IH^{\olq}_* (c^\circ \Sigma))_j 
\cong (H_* (L)\otimes \tau_{\leq a} H_* (\Sigma))_j 
= \bigoplus_{t=0}^a H_{j-t} (L)\otimes H_t (\Sigma).
\]
and
\[
I_{j-1} \cong (H_* (L)\otimes IH^{\olq}_* (c^\circ \Sigma))_{j-1} 
\cong (H_* (L)\otimes \tau_{\leq a} H_* (\Sigma))_{j-1} 
= \bigoplus_{t=0}^a H_{j-1-t} (L)\otimes H_t (\Sigma).
\]
Similarly for $\olq'$,
\[
J_j \cong (H_* (L)\otimes IH^{\olq'}_* (c^\circ \Sigma))_j 
\cong (H_* (L)\otimes \tau_{\leq a-1} H_* (\Sigma))_j 
= \bigoplus_{t=0}^{a-1} H_{j-t} (L)\otimes H_t (\Sigma),
\]
\[
J_{j-1} \cong (H_* (L)\otimes IH^{\olq'}_* (c^\circ \Sigma))_{j-1} 
\cong \bigoplus_{t=0}^{a-1} H_{j-1-t} (L)\otimes H_t (\Sigma),
\]
see Figure \ref{fig.dia}.
\begin{figure} 
\resizebox{8cm}{7cm}{\includegraphics{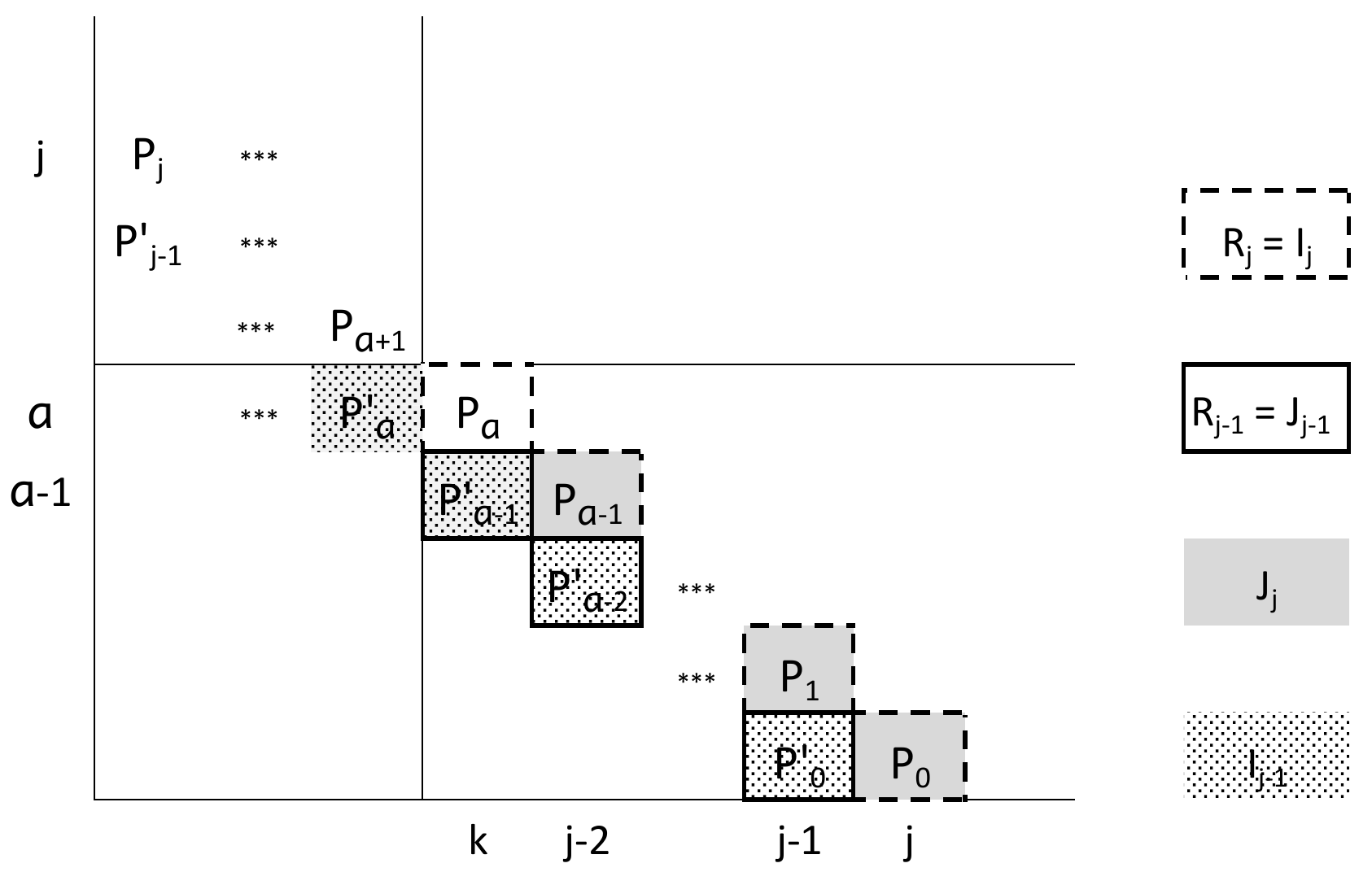}}
\caption{\label{fig.dia} Local K\"unneth factor truncations.}
\end{figure}
This concludes the local calculations of groups.

We commence the determination of various local maps near the singularities.
Let $\gamma^{\loc}_i: I_i \to J_i$ the canonical map.
Using the collar associated to the boundary of the blowup, the open inclusion $L\times \Sigma \times (0,1)
\hookrightarrow M$ induces a map
\[ \beta^T: B_* \longrightarrow T_*, \]
while the open inclusion $L\times \Sigma \times (0,1) \hookrightarrow L\times c^\circ \Sigma$
induces maps 
\[   \beta^I: B_* \longrightarrow I_*,~ \beta^J: B_* \longrightarrow J_* \]
such that
\begin{equation} \label{equ.biglbj}
\xymatrix{
B_* \ar[r]^{\beta^I} \ar[rd]_{\beta^J} & I_* \ar[d]^{\gamma^{\loc}} \\
& J_*
} 
\end{equation}
commutes.
The canonical inclusion $L\times \Sigma \hookrightarrow \cone (f\times \id_\Sigma)$
induces a map
\[ \beta^R: B_* \longrightarrow R_*. \]
Write $P_t = H_{j-t} (L) \otimes H_t (\Sigma)$ and let
$[j,a]: \bigoplus_{t=0}^j P_t \to \bigoplus_{t=0}^a P_t$
be the standard projection if $j>a,$ the identity if $j\leq a$.
The diagram
\[ \xymatrix{
B_j \ar[r]^{\beta^I_j} & I_j \\
 \bigoplus_{t=0}^j P_t \ar[u]^{\times}_{\cong} \ar[r]_{[j,a]} &  \bigoplus_{t=0}^a P_t, 
  \ar[u]^{\cong}_{\times}
} \]
where the vertical isomorphisms are given by the cross product, commutes.
For $j>0$, let $\rho^{\loc}_j: I_j \to R_j$ be the unique isomorphism such that
\[ \xymatrix{
I_j \ar@{..>}[r]^{\rho^{\loc}_j} & R_j \\
 \bigoplus_{t=0}^a P_t \ar[u]^{\times}_{\cong} \ar@{=}[r] &  \bigoplus_{t=0}^a P_t 
  \ar[u]^{\cong}_{\times}
} \]
commutes, using Lemma \ref{lem.computationrj}.
Since
\[ \xymatrix{
B_j \ar[r]^{\beta^R_j} & R_j \\
 \bigoplus_{t=0}^j P_t \ar[u]^{\times}_{\cong} \ar[r]_{[j,a]} &  \bigoplus_{t=0}^a P_t 
  \ar[u]^{\cong}_{\times}
} \]
commutes, we know that
\begin{equation} \label{equ.birj}
 \xymatrix@R=40pt@C=40pt{
B_j \ar[r]^{\beta^I_j} \ar[rd]_{\beta^R_j} & I_j \ar[d]^{\rho^{\loc}_j}_{\cong} \\
& R_j
} 
\end{equation}
commutes as well.
Write $P'_t = H_{j-1-t} (L) \otimes H_t (\Sigma)$ and let
$[j-1,a]: \bigoplus_{t=0}^{j-1} P'_t \to \bigoplus_{t=0}^a P'_t$
be defined as above.
For $j>1$, let $\rho^{\loc}_{j-1}: I_{j-1} \to R_{j-1}$ be the unique epimorphism such that
\[ \xymatrix{
I_{j-1} \ar@{..>>}[r]^{\rho^{\loc}_{j-1}} & R_{j-1} \\
 \bigoplus_{t=0}^a P'_t \ar[u]^{\times}_{\cong} \ar@{->>}[r]^{[a,a-1]} &  
 \bigoplus_{t=0}^{a-1} P'_t 
  \ar[u]^{\cong}_{\times}
} \]
commutes, using Lemma \ref{lem.computationrj}.
Under the cross product, the commutative diagram
\[ \xymatrix{
\bigoplus_{t=0}^{j-1} P'_t \ar[r]^{[j-1,a]} \ar[rd]_{[j-1,a-1]} & \bigoplus_{t=0}^{a} P'_t
   \ar@{->>}[d]^{[a,a-1]} \\
& \bigoplus_{t=0}^{a-1} P'_t
} \]
corresponds to
\begin{equation} \label{equ.birjm1}
\xymatrix{
B_{j-1} \ar[r]^{\beta^I_{j-1}} \ar[rd]_{\beta^R_{j-1}} &
  I_{j-1} \ar@{->>}[d]^{\rho^{\loc}_{j-1}} \\
& R_{j-1},
} 
\end{equation}
which therefore also commutes.

For any $j$, the diagrams
\[ \xymatrix{
I_j \ar[r]^{\gamma^{\loc}_j} & J_j \\
 \bigoplus_{t=0}^a P_t \ar[u]^{\times}_{\cong} \ar[r]_{[a,a-1]} &  \bigoplus_{t=0}^{a-1} P_t 
  \ar[u]^{\cong}_{\times}
} \]
and
\[ \xymatrix{
I_{j-1} \ar[r]^{\gamma^{\loc}_{j-1}} & J_{j-1} \\
 \bigoplus_{t=0}^a P'_t \ar[u]^{\times}_{\cong} \ar[r]_{[a,a-1]} &  \bigoplus_{t=0}^{a-1} P'_t 
  \ar[u]^{\cong}_{\times}
} \]
commute, showing that both $\gamma^{\loc}_j$ and $\gamma^{\loc}_{j-1}$ are surjective.
For $j>1$, let $R_{j-1} \to J_{j-1}$ be the unique isomorphism such that
\[ \xymatrix{
R_{j-1} \ar[r]^{\cong} & J_{j-1} \\
 \bigoplus_{t=0}^{a-1} P'_t \ar[u]^{\times}_{\cong} \ar@{=}[r] &  
 \bigoplus_{t=0}^{a-1} P'_t 
  \ar[u]^{\cong}_{\times}
} \]
commutes. Then, since $\rho^{\loc}_{j-1}$ and $\gamma^{\loc}_{j-1}$ are both under the
K\"unneth isomorphism given by
the projection $[a,a-1],$ the diagram
\begin{equation} \label{equ.ijr}
\xymatrix{
I_{j-1} \ar@{->>}[rr]^{\gamma^{\loc}_{j-1}} \ar@{->>}[rd]_{\rho^{\loc}_{j-1}} & & J_{j-1} \\
& R_{j-1} \ar[ru]_\cong &
} 
\end{equation}
commutes ($j>1$).
\begin{lemma} \label{lem.kerbjiskerbr}
When $j>1$, the identity $\ker \beta^J_{j-1} = \ker \beta^R_{j-1}$ holds in $B_{j-1}$.
\end{lemma}
\begin{proof}
By diagram (\ref{equ.ijr}), in the case $j>1$ there is an isomorphism
$\nu: R_{j-1} \to J_{j-1}$ such that $\nu \rho^{\loc}_{j-1} = \gamma^{\loc}_{j-1}$.
According to diagram (\ref{equ.birjm1}), $\beta^R_{j-1} = \rho^{\loc}_{j-1} \beta^I_{j-1}$.
Furthemore, $\beta^J_{j-1} = \gamma^{\loc}_{j-1} \beta^I_{j-1}$ by diagram (\ref{equ.biglbj}).
Hence $\beta^R_{j-1} (x) = \rho^{\loc}_{j-1} \beta^I_{j-1} (x)$ vanishes if and only if
\[ \nu \rho^{\loc}_{j-1} \beta^I_{j-1} (x) = \gamma^{\loc}_{j-1} \beta^I_{j-1} (x)=\beta^J_{j-1} (x) \]
vanishes.
\end{proof}
This concludes our investigation of local maps.\\

We move on to global arguments.
The open cover $CT(X)=M \cup (L\times c^\circ \Sigma)$ with 
$M\cap L\times c^\circ \Sigma = L\times \Sigma \times (0,1)\simeq L\times \Sigma$
yields a Mayer-Vietoris sequence for intersection homology
\[ B_j \stackrel{\beta_j}{\longrightarrow} T_j \oplus I_j \stackrel{\theta}{\longrightarrow} IH^{\olq}_j (CT(X))
 \stackrel{\partial_*}{\longrightarrow} B_{j-1} 
  \stackrel{\beta_{j-1}}{\longrightarrow} T_{j-1} \oplus I_{j-1}, \]
see (\ref{equ.ihmvgeneric}).
Similarly, there is such a sequence for perversity $\olq'$:
\[ B_j \stackrel{\beta'_j}{\longrightarrow} T_j \oplus J_j 
  \stackrel{\theta'}{\longrightarrow} IH^{\olq'}_j (CT(X))
 \stackrel{\partial_*}{\longrightarrow} B_{j-1} 
 \stackrel{\beta'_{j-1}}{\longrightarrow} T_{j-1} \oplus J_{j-1}. \]
The canonical map from perversity $\olq$ to $\olq'$ induces a commutative diagram
\begin{equation} \label{equ.ihqihqprime}
\xymatrix{
B_j \ar[r]^{\beta_j} \ar@{=}[d] & T_j \oplus I_j \ar[r]^\theta \ar[d]^{\id \oplus \gamma^{\loc}_j} 
  &  IH^{\olq}_j (CT(X))
 \ar[r]^{\partial_*} \ar[d]^\gamma &  
 B_{j-1} \ar[r]^{\beta_{j-1}} \ar@{=}[d] &  T_{j-1} \oplus I_{j-1} \ar[d]^{\id \oplus \gamma^{\loc}_{j-1}} \\
B_j \ar[r]^{\beta'_j} & T_j \oplus J_j \ar[r]^{\theta'} &  IH^{\olq'}_j (CT(X))
 \ar[r]^{\partial_*} &  B_{j-1} \ar[r]^{\beta'_{j-1}} &  T_{j-1} \oplus J_{j-1}. 
} 
\end{equation}
The subset 
\[ U= ((0,1] \times L\times \Sigma) \cup_{\{ 1 \} \times L\times \Sigma} \cone (f\times \id_\Sigma)
 \simeq \cone (f\times \id_\Sigma) \]
is open in the intersection space $\ip X$.
The open cover $I^{\olp} X=M \cup U$ with 
$M\cap U= L\times \Sigma \times (0,1)\simeq L\times \Sigma$
yields a Mayer-Vietoris sequence
\[ B_j \stackrel{\beta''_j}{\longrightarrow} T_j \oplus R_j 
 \stackrel{\theta''}{\longrightarrow} HI^{\olp}_j (X)
 \stackrel{\partial_*}{\longrightarrow} B_{j-1} 
 \stackrel{\beta''_{j-1}}{\longrightarrow} T_{j-1} \oplus R_{j-1}. \]
This is the standard Mayer-Vietoris sequence for singular homology of topological spaces.

We shall prove Theorem \ref{ihhihom} first for all $j>1$.
Using the commutative diagrams (\ref{equ.birj}) and (\ref{equ.birjm1}), we obtain the following commutative
diagram with exact rows:
\begin{equation} \label{equ.MVihtohi}
\xymatrix{
B_j \ar[r]^{\beta_j} \ar@{=}[d] & T_j \oplus I_j \ar[r]^\theta \ar[d]^{\id \oplus \rho^{\loc}_j}_\cong 
  &  IH^{\olq}_j (CT(X))
 \ar[r]^{\partial_*} &  B_{j-1} \ar[r]^{\beta_{j-1}} \ar@{=}[d] &  
  T_{j-1} \oplus I_{j-1} \ar@{->>}[d]^{\id \oplus \rho^{\loc}_{j-1}} \\
B_j \ar[r]^{\beta''_j} & T_j \oplus R_j \ar[r]^{\theta''} &  HI^{\olp}_j (X)
 \ar[r]^{\partial_*} &  B_{j-1} \ar[r]^{\beta''_{j-1}} &  T_{j-1} \oplus R_{j-1}. 
} 
\end{equation}
There exists a (nonunique) map 
$\rho:  IH^{\olq}_j (CT(X)) \to HI^{\olp}_j (X)$ which fills in diagram 
(\ref{equ.MVihtohi}) commutatively, see e.g. \cite[Lemma 2.46]{Ba-HI}.
By the four-lemma, $\rho$ is a monomorphism.
This shows that $HI^{\olp}_j (X)$ contains  $IH^{\olq}_j (CT(X))$ as a subspace.
Hence the theorem will follow from:
\begin{prop}  \label{prop.isocokrhocokgam}
If $j>1$, there is an isomorphism $\cok \rho \cong \cok \gamma$.
\end{prop}
\begin{proof}
Let us determine the cokernel of $\gamma$.
Let $\gamma_B$ be the restriction of $\id: B_{j-1} \to B_{j-1}$ to
$\gamma_B: \ker \beta_{j-1} \to \ker \beta'_{j-1},$ i.e. $\gamma_B$ is the inclusion
$\ker \beta_{j-1} \subset \ker \beta'_{j-1}$.
Let $\gamma_\theta: \im \theta \to \im \theta'$ be obtained by restricting $\gamma$.
Applying the snake lemma to the commutative diagram
\[ \xymatrix{
0 \ar[r] & \im \theta \ar[r] \ar[d]^{\gamma_\theta} &
 IH^{\olq}_j (CT(X)) \ar[r]^{\partial_*} \ar[d]_\gamma & 
  \ker \beta_{j-1} \ar[r] \ar[d]_{\gamma_B} & 0 \\
0 \ar[r] & \im \theta' \ar[r]  &
 IH^{\olq'}_j (CT(X)) \ar[r]^{\partial_*} & 
  \ker \beta'_{j-1} \ar[r]  & 0
} \]
yields an exact sequence
\[ 0 \to \ker \gamma_\theta \to \ker \gamma \to \ker \gamma_B \to
 \cok \gamma_\theta \to \cok \gamma \to \cok \gamma_B \to 0. \]
Since $\ker \gamma_B =0,$ we can extract the short exact sequence
\[ 0 \to \cok \gamma_\theta \to \cok \gamma \to \cok \gamma_B \to 0. \]
As $\gamma^{\loc}_j$ is surjective, the diagram
\[ \xymatrix{
T_j \oplus I_j \ar@{->>}[r]^\theta \ar@{->>}[d]_{\id \oplus \gamma^{\loc}_j} & \im \theta \ar[d]^{\gamma_\theta} \\
T_j \oplus J_j \ar@{->>}[r]^{\theta'} & \im \theta' 
} \]
shows that $\gamma_\theta$ is also surjective and thus $\cok \gamma_\theta =0$.
Therefore, we obtain an isomorphism
\[ \cok \gamma \stackrel{\cong}{\longrightarrow}
 \cok \gamma_B = \frac{\ker \beta'_{j-1}}{\ker \beta_{j-1}}
 = \frac{\ker \beta^T_{j-1} \cap \ker \beta^J_{j-1}}{\ker \beta^T_{j-1} \cap \ker \beta^I_{j-1}},
\]
since $\beta_{j-1} = (\beta^T_{j-1}, \beta^I_{j-1})$ and
$\beta'_{j-1} = (\beta^T_{j-1}, \beta^J_{j-1}).$ 

In a similar manner, we determine the cokernel of $\rho$.
Let $\rho_B$ be the restriction of $\id: B_{j-1} \to B_{j-1}$ to
$\rho_B: \ker \beta_{j-1} \to \ker \beta''_{j-1},$ i.e. $\rho_B$ is the inclusion
$\ker \beta_{j-1} \subset \ker \beta''_{j-1}$.
Let $\rho_\theta: \im \theta \to \im \theta''$ be obtained by restricting $\rho$.
Applying the snake lemma to the commutative diagram
\[ \xymatrix{
0 \ar[r] & \im \theta \ar[r] \ar[d]^{\rho_\theta} &
 IH^{\olq}_j (CT(X)) \ar[r]^{\partial_*} \ar[d]_\rho & 
  \ker \beta_{j-1} \ar[r] \ar[d]_{\rho_B} & 0 \\
0 \ar[r] & \im \theta'' \ar[r]  &
 HI^{\olp}_j (X) \ar[r]^{\partial_*} & 
  \ker \beta''_{j-1} \ar[r]  & 0
} \]
yields an exact sequence
\[ 0 \to \ker \rho_\theta \to \ker \rho \to \ker \rho_B \to
 \cok \rho_\theta \to \cok \rho \to \cok \rho_B \to 0. \]
Since $\ker \rho_B =0,$ we can extract the short exact sequence
\[ 0 \to \cok \rho_\theta \to \cok \rho \to \cok \rho_B \to 0. \]
As $\rho^{\loc}_j$ is an isomorphism, the diagram
\[ \xymatrix{
T_j \oplus I_j \ar@{->>}[r]^\theta \ar[d]_{\id \oplus \rho^{\loc}_j}^{\cong} & \im \theta \ar[d]^{\rho_\theta} \\
T_j \oplus R_j \ar@{->>}[r]^{\theta''} & \im \theta'' 
} \]
shows that $\rho_\theta$ is also surjective and thus $\cok \rho_\theta =0$.
Therefore, we obtain an isomorphism
\[ \cok \rho \stackrel{\cong}{\longrightarrow}
 \cok \rho_B = \frac{\ker \beta''_{j-1}}{\ker \beta_{j-1}}.
\]
As $\beta''_{j-1} = (\beta^T_{j-1}, \beta^R_{j-1}),$ we have
\[ \ker \beta''_{j-1} = \ker \beta^T_{j-1} \cap \ker \beta^R_{j-1}. \]
By Lemma \ref{lem.kerbjiskerbr}, $\ker \beta^J_{j-1} = \ker \beta^R_{j-1}$ and hence
\[ \cok \rho \cong \frac{\ker \beta^T_{j-1} \cap \ker \beta^R_{j-1}}{\ker \beta^T_{j-1} \cap \ker \beta^I_{j-1}}
 =\frac{\ker \beta^T_{j-1} \cap \ker \beta^J_{j-1}}{\ker \beta^T_{j-1} \cap \ker \beta^I_{j-1}}
 \cong \cok \gamma. \]
\end{proof}

Now assume that $j=1$.
We need to consider the subcases $k\leq 0$, $k=1$ and $k>1$ separately.
We start with $k\leq 0$. In principle, we shall again use a diagram of the
shape (\ref{equ.MVihtohi}), but the definition of $\rho^{\loc}_{j-1}$ changes.
In this case, $L_{<k}=\varnothing,$ 
$\cone (f\times \id_\Sigma)=(L\times \Sigma)^+$,
$R_0 = \real \oplus H_0 (L)\otimes H_0 (\Sigma),$
$a\geq 1$ and $I_0 \cong J_0 \cong H_0 (L)\otimes H_0 (\Sigma)$.
The maps of the commutative diagram
\[ \xymatrix@C=50pt@R=50pt{
B_0 \ar[r]^{\beta^I_0}_\cong \ar[rd]_{\beta^J_0}^\cong & 
    I_0 \ar[d]^{\gamma^{\loc}_0}_\cong \\
& J_0
}  \]
are all isomorphisms.
The map $L\times \Sigma \hookrightarrow \cone (f\times \id_\Sigma),$
which induces $\beta^R,$ is the inclusion 
$L\times \Sigma \hookrightarrow (L\times \Sigma)^+$ and thus 
$\beta^R_0: B_0 \to R_0$ is the standard inclusion
$H_0 (L)\otimes H_0 (\Sigma) \to \real \oplus H_0 (L)\otimes H_0 (\Sigma).$
Let $\rho^{\loc}_{j-1} = \rho^{\loc}_0: I_0 \to R_0$ be the unique monomorphism such that
\[ \xymatrix{
B_{0} \ar[r]^{\beta^I_{0}}_\cong \ar@{^{(}->}[rd]_{\beta^R_{0}} &
  I_{0} \ar@{..>}[d]^{\rho^{\loc}_{0}} \\
& R_{0}
} \]
commutes. (Note that for $j\geq 2,$ $\rho^{\loc}_{j-1}$ was known to be surjective,
which is not true here.) Diagram (\ref{equ.MVihtohi}) becomes
\[
\xymatrix{
B_1 \ar[r]^{\beta_1} \ar@{=}[d] & T_1 \oplus I_1 \ar[r]^\theta \ar[d]^{\id \oplus \rho^{\loc}_1}_\cong 
  &  IH^{\olq}_1 (CT(X))
 \ar[r]^{\partial_*} &  B_{0} \ar[r]^{\beta_{0}} \ar@{=}[d] &  
  T_{0} \oplus I_{0} \ar@{^{(}->}[d]^{\id \oplus \rho^{\loc}_{0}} \\
B_1 \ar[r]^{\beta''_1} & T_1 \oplus R_1 \ar[r]^{\theta''} &  HI^{\olp}_1 (X)
 \ar[r]^{\partial_*} &  B_{0} \ar[r]^{\beta''_{0}} &  T_{0} \oplus R_{0}. 
} 
\] 
There exists a (nonunique) map 
$\rho:  IH^{\olq}_1 (CT(X)) \to HI^{\olp}_1 (X)$ which fills in the diagram 
commutatively.
By the five-lemma, $\rho$ is an isomorphism.
To establish the theorem, it remains to be shown that $\cok \gamma$ vanishes.
Since diagram (\ref{equ.ihqihqprime}) is available for any $j$, the argument
given in the proof of Proposition \ref{prop.isocokrhocokgam} still applies to give an isomorphism
\[ \cok \gamma \cong \frac{\ker \beta'_{0}}{\ker \beta_{0}} =
\frac{\ker \beta^T_{0} \cap \ker \beta^J_{0}}{\ker \beta^T_{0} \cap \ker \beta^I_{0}}.
\]
Since $\beta^I_0$ and $\beta^J_0$ are isomorphisms, we deduce that $\cok \gamma =0$,
as was to be shown. This concludes the case $k\leq 0$. \\

We proceed to the case $k=1$ (and $j=1$).
By Lemma \ref{lem.computationrj}, $R_0 =\real,$ generated by the cone vertex.
We have $a=0$ and thus still $I_0 = H_0 (L)\otimes H_0 (\Sigma),$ but
$J_0 =0$. Therefore, $\ker \beta^J_0 =B_0$.
The map $\beta^R_0: B_0 \to R_0$ can be identified with the
augmentation map $\epsilon: B_0 \to \real$, a surjection.
Let $\rho^{\loc}_{j-1} = \rho^{\loc}_0: I_0 \to R_0$ be the unique epimorphism such that
\[ \xymatrix{
B_{0} \ar[r]^{\beta^I_{0}}_\cong \ar@{->>}[rd]_{\epsilon =\beta^R_{0}} &
  I_{0} \ar@{..>}[d]^{\rho^{\loc}_{0}} \\
& R_{0} =\real
} \]
commutes.
There exists a map $\rho$ filling in diagram (\ref{equ.MVihtohi}) commutatively.
Such a $\rho$ is then injective.
Using arguments from the proof of Proposition \ref{prop.isocokrhocokgam},
we have
\[ \cok \gamma \cong 
\frac{\ker \beta^T_{0} \cap \ker \beta^J_{0}}{\ker \beta^T_{0} \cap \ker \beta^I_{0}}
 = \ker \beta^T_0 \cap B_0 = \ker \beta^T_0
\]
and
\[ \cok \rho \cong 
\ker \beta^T_{0} \cap \ker \beta^R_{0}
 = \ker \beta^T_0 \cap \ker \epsilon.
\]
We recall from elementary algebraic topology:
\begin{lemma} \label{lem.epsheps}
If $A$ and $B$ are topological spaces and $h:A\to B$ a continuous map,
then the diagram
\[ \xymatrix{
H_0 (A) \ar[r]^{h_*} \ar[rd]_{\epsilon} & H_0 (B) \ar[d]^{\epsilon} \\
& \real
} \]
commutes. In particular, $\ker h_* \subset \ker (\epsilon: H_0 (A)\to \real)$.
\end{lemma}
Applying this lemma to $h_* = \beta^T_0$, we have 
$\ker \beta^T_0 \subset \ker \epsilon$ and thus
$\cok \rho \cong \ker \beta^T_0 \cong \cok \gamma$. This concludes the
proof in the case $k=1$. \\

When $k>1$ (and $j=1$), then $R_0 =\real$ (Lemma \ref{lem.computationrj}),
$a<0,$ and $I_0 = J_0 =0$. As in the case $k=1$, the map
$\beta^R_0: B_0 \to R_0$ can be identified with the
augmentation epimorphism $\epsilon: B_0 \to \real$, but
this time, there does not exist a map
$\rho^{\loc}_0$ such that $\rho^{\loc}_0 \beta^I_0 = \beta^R_0$.
We must therefore argue differently.
By exactness and since $\beta^I_0 =0,$ we have
\[ \im (\partial_*: IH^{\olq}_1 (CT(X))\to B_0) = \ker (\beta^T_0, \beta^I_0) 
  = \ker \beta^T_0. \]
Also,
\[ \im (\partial_*: HI^{\olp}_1 (X)\to B_0) = \ker (\beta^T_0, \beta^R_0) 
  = \ker \beta^T_0 \cap \ker \epsilon. \]
By Lemma \ref{lem.epsheps}, $\ker \beta^T_0 \subset \ker \epsilon$ and hence
\[ \im (\partial_*: IH^{\olq}_1 (CT(X))\to B_0) =
   \im (\partial_*: HI^{\olp}_1 (X)\to B_0). \]
For the diagram
\[
\xymatrix{
B_1 \ar[r]^{\beta_1}  & T_1 \oplus I_1 \ar[r]^\theta  
  &  IH^{\olq}_1 (CT(X))
 \ar[r]^{\partial_*} &  \im \partial_* \ar[r]  & 0 \\
B_1 \ar[r]^{\beta''_1} \ar@{=}[u] & T_1 \oplus R_1 \ar[r]^{\theta''} 
  \ar[u]_{\id \oplus (\rho^{\loc}_1)^{-1}}^\cong
&  HI^{\olp}_1 (X)
 \ar[r]^{\partial_*} &  \im \partial_* \ar[r] \ar@{=}[u] &  0, \ar@{=}[u],
} 
\] 
there exists a (nonunique) map 
$\rho':  HI^{\olp}_1 (X) \to IH^{\olq}_1 (CT(X))$ which fills in the diagram 
commutatively.
By the five-lemma, $\rho'$ is an isomorphism. The cokernel of $\gamma$,
\[ \cok \gamma \cong 
\frac{\ker \beta^T_{0} \cap \ker \beta^J_{0}}{\ker \beta^T_{0} \cap \ker \beta^I_{0}} =
 \frac{\ker \beta^T_{0} \cap B_0}{\ker \beta^T_{0} \cap B_0} =0,
\]
vanishes and thus the theorem holds in this case as well. This finishes the proof
for $j=1$. \\

It remains to establish Theorem \ref{ihhihom} for $j=0$. 
We shall write $\widetilde{B}_*,$ $\widetilde{T}_*,$
$\widetilde{R}_*$ for the reduced homology groups.
We have $a=-k$,
\[ I_0 \cong \begin{cases}
  H_0 (L) \otimes H_0 (\Sigma),& k\leq 0 \\
  0,& k>0
 \end{cases}
\]
and by Lemma \ref{lem.computationrj},
\[ \widetilde{R}_0 \cong \begin{cases}
  H_0 (L) \otimes H_0 (\Sigma),& k\leq 0 \\
  0,& k>0.
 \end{cases}
\]
Thus $I_0$ and $\widetilde{R}_0$ are abstractly isomorphic.
 Recall that for a topological space $A$, the reduced homology $\redh_0 (A)$
is the kernel of the augmentation map $\epsilon: H_0 (A)\to \real$, so that there
is a short exact sequence
\[ 0\longrightarrow \redh_0 (A) \stackrel{\iota}{\longrightarrow} H_0 (A)
 \stackrel{\epsilon}{\longrightarrow} \real \longrightarrow 0. \]
Let $\widetilde{\beta}^R_0: \widetilde{B}_0 \to \widetilde{R}_0$ be the
map induced by $\beta^R_0$ between the kernels of the respective augmentation maps.
If $k\leq 0$, then $L_{<k} =\varnothing$ and the exact sequence
\[ H_0 (L_{<k} \times \Sigma) \longrightarrow
  H_0 (L\times \Sigma) \longrightarrow
 \widetilde{R}_0 \longrightarrow 0 \]
shows that $B_0 \to \widetilde{R}_0$ is an isomorphism.
Since the composition 
\[ \widetilde{B}_0 \stackrel{\iota}{\hookrightarrow} B_0 \stackrel{\cong}{\longrightarrow}
  \widetilde{R}_0 \]
is $\widetilde{\beta}^R_0$,
we deduce that $\widetilde{\beta}^R_0$ is injective for $k\leq 0$.
Inverting $\widetilde{\beta}^R_0$ on its image and composing with $\beta^I_0$,
then extending to an isomorphism using $\dim I_0 = \dim \widetilde{R}_0$, we obtain
an isomorphism $\kappa: \widetilde{R}_0 \to I_0$ such that
\begin{equation}  \label{equ.kappa}
\xymatrix{
\widetilde{B}_0 \ar[r]^{\widetilde{\beta}^R_0} \ar@{^{(}->}[d]_{\iota} &
  \widetilde{R}_0 \ar@{..>}[d]^\kappa \\
B_0 \ar[r]_{\beta^I_0} & I_0
} 
\end{equation}
commutes. When $k>0$, let $\kappa: \widetilde{R}_0 \to I_0$ be the zero map
(an isomorphism). Then diagram (\ref{equ.kappa}) commutes also in this case.
As for the above open cover $\ip X = M\cup U,$ the intersection
$M\cap U = L\times \Sigma \times (0,1)$ is not empty, so there is a
Mayer-Vietoris sequence on reduced homology:
\[ \widetilde{B}_0 \stackrel{\beta''_0}{\longrightarrow} 
  \widetilde{T}_0 \oplus \widetilde{R}_0 
 \stackrel{\theta''}{\longrightarrow} \redh I^{\olp}_0 (X)
 \longrightarrow 0. \]
Using $\kappa$, we get the following commutative diagram with exact rows:
\[ \xymatrix{
 \widetilde{B}_0 \ar[r]^{\widetilde{\beta}''_0} \ar@{^{(}->}[d]_{\iota} & 
  \widetilde{T}_0 \oplus \widetilde{R}_0 
   \ar[r]^{\theta''} \ar@{^{(}->}[d]_{\iota \oplus \kappa} &
 \redh I^{\olp}_0 (X) \ar[r] & 0 \\
 B_0 \ar[r]^{\beta_0} & 
  T_0 \oplus I_0 
   \ar[r]^{\theta} &
  IH^{\olq}_0 (CT(X)) \ar[r] & 0, \\
} \] 
from which we infer that
\[  \redh I^{\olp}_0 (X) \cong \cok \widetilde{\beta}''_0,~
  IH^{\olq}_0 (CT(X)) \cong \cok \beta_0. \]
Let $T_0 \oplus I_0 \to \real$ be the composition of the standard projection
$T_0 \oplus I_0 \to T_0$ with the augmentation $\epsilon: T_0 \to \real$.
Applying the snake lemma to the commutative diagram
\[ \xymatrix{
0 \ar[r] &
\widetilde{B}_0 \ar[r]^{\iota} \ar[d]^{\widetilde{\beta}''_0} &
B_0 \ar[r]^\epsilon \ar[d]^{\beta_0} &
\real \ar[r] \ar@{=}[d] & 0 \\
0 \ar[r] &
\widetilde{T}_0 \oplus \widetilde{R}_0 \ar[r]^{\iota \oplus \kappa}  &
T_0 \oplus I_0 \ar[r]  &
\real \ar[r] & 0 \\
} \]
we arrive at the exact sequence
\[ 0\to \ker \widetilde{\beta}''_0 \to \ker \beta_0 \to \ker \id_\real
 \to \cok \widetilde{\beta}''_0 \to \cok \beta_0 \to \cok \id_\real \to 0. \]
As $\ker \id_\real =0$ and $\cok \id_\real =0$, we obtain
an isomorphism $\cok \widetilde{\beta}''_0 \cong \cok \beta_0,$
i.e. $\redh I^{\olp}_0 (X)\cong IH^{\olq}_0 (CT(X))$.
It remains to be shown that 
$\gamma: IH^{\olq}_0 (CT(X)) \to IH^{\olq'}_0 (CT(X))$ is surjective.
This follows from the surjectivity of $\gamma^{\loc}_0$ and diagram
(\ref{equ.ihqihqprime}).

\bigskip

\subsection{Example}\label{example1}
We may consider the following example to illustrate Theorem \ref{ihhihom}.  Consider the two-sphere, as a stratified
space, thought of as the suspension of $S^1$.  So the two poles are the ``singular" stratum, with link $L=S^1$,
and we will denote these by $\pm \pt$.  Now 
take $X=S^2 \times T^2$ with the induced stratification, $\Sigma = \{\pm \pt \} \times T^2 \subset X$.  The codimension
of $\Sigma$ in $X$ is 2, and any standard perversity takes $\olp(2)=0$. We will first calculate $H_*(I^\olp X)$.
Let $\olM = X - N(\Sigma)$, where $N$ is an open normal neighborhood of $\Sigma$.  Note that 
$\del\olM \cong S^1 \times \{ \pm \pt \} \times T^2$.
The cutoff degree here is $k=1-\olp(2)=1$, so $I^\olp X= \olM \cup_g c(L_{< 1} \times T^2)$, where $L_{<1} \subset L=S^1$.  
For any path connected space, $L_{<1}$ is just a point $e^0$ in the space.  Thus $g$ is the inclusion map
$g: e^0 \times \{ \pm \pt \} \times T^2 \hookrightarrow \olM$.   
The reduced homology is given by
\[
\widetilde{HI}_*^\olp(X) = \widetilde{H}_*(I^\olp X) = \widetilde{H}_*(\olM, e^0 \times \{\pm \pt\} \times T^2).
\]
Then using the relative exact sequence on homology, we can calculate this as:
\[
\widetilde{HI}_i^\olp(X) \cong \left\{
\begin{array}{ll}
0 & i=0 \\
\mathbb{R}^2 & i=1 \\ 
\mathbb{R}^4 & i=2 \\
\mathbb{R}^2 & i=3 \\ 
0 & i=4.
\end{array}
\right. 
\]
Theorem \ref{ihhihom} states that
there is a relationship between the $\widetilde{HI}_i^\olp(X)$ we just calculated and intersection homology groups
for $CT(X)$.  The conifold transition of $X$ is the suspension of $T^2$ times $S^1$:  $CT(X) = S(T^2) \times S^1$.  This has singular stratum $B= \{\pm \pt\} \times S^1$ with link $F=T^2$, so the codimension of $B$ is 3.  
This means there are two possible standard perversities $\olm$ (lower middle) and $\oln$ (upper middle), where 
$\olm (3)=0$ and $\oln (3)=1$.  
The intersection homology groups of $CT(X)$ for the perversities $\olm$ and $\oln$ are calculated in \cite{Ba-HI}, p. 79, 
which also indicates the generators of the classes.   

In order to illustrate Theorem \ref{ihhihom}, we
also need to understand $IH_\olq^*(CT(X))$ where $\olq$ is an extended perversity, as in \cite{FH}, and
we need to understand the maps between
consecutive perversity intersection homology groups to calculate the groups $IG^{(k)}_*(CT(X))$.  
By Proposition \ref{prop.ihctxextperv},
$IH^{\olq}_*(CT(X)) = H_*(\olM)$ for $\olq(3)<0$ and 
$IH^{\olq}_*(CT(X)) =H_*(\olM, \del \olM)$ for $\olq(3)\geq 2.$
If $\olq(3) < \olq'(3),$ then there is a natural map 
\[
IH^\olq_*(CT(X)) \to IH^{\olq'}_*(CT(X)),
\]
since any cycle satisfying the more restrictive condition given by $\olq$ will in particular also satisfy the less restrictive
condition given by $\olq'$.  This is the map that appears in the definition of $IG^{(k)}_*(CT(X))$. Now we can create the following
table that will allow us to calculate these groups.
\begin{table}[ht]
\caption{$IH^\olq_j(CT(X))$} 
\centering 
\begin{tabular}{c | c c c c c c c} 
\hline\hline 
$j \, \backslash \, \olq(3)$ & $-1$ & $\to$  & $0$ & $\to$ & $1$ & $\to$ & $2$  \\ [1ex] 
\hline 
0 &$\mathbb{R}$ & $\cong$ &$\mathbb{R}$&$\cong$ &$\mathbb{R}$& $\stackrel{0}{\rightarrow} $&0 \\ [1ex]
1 &$\mathbb{R}^3$ & $\cong$ & $\mathbb{R}^3$ & $\onto$ & $\mathbb{R}$ & $\stackrel{0}{\rightarrow}$ &$\mathbb{R}$\\[1ex]
2 &$\mathbb{R}^3$&$\onto$&$\mathbb{R}^2$&$\stackrel{0}{\rightarrow}$&$\mathbb{R}^2$&$\hookrightarrow$&$\mathbb{R}^3$ \\[1ex]
3 &$\mathbb{R}$  &$\stackrel{0}{\rightarrow}$&$\mathbb{R}$&$\hookrightarrow$&$\mathbb{R}^3$&$\cong$&$\mathbb{R}^3$\\[1ex]
4 &0  &$\stackrel{0}{\rightarrow}$&$\mathbb{R}$&$\cong$&$\mathbb{R}$&$\cong$&$\mathbb{R}$\\ [1ex] 
\hline 
\end{tabular}
\label{table:nonlin} 
\end{table}

We get, for example,
\begin{eqnarray*}
IG^{(2)}_1 (CT(X)) &\cong& \frac{IH^{\olq(3)=1}_1(CT(X)) \oplus IH^{\olq'(3)=2}_1(CT(X))}{ {\rm Image}\left(IH^{\olq(3)=1}_1(CT(X)) \to H^{\olq'(3)=2}_1(CT(X))\right)} \\
&\cong&  \frac{\R \oplus \R}{0}\\
&\cong & \R^2.
\end{eqnarray*}
Collecting the relevant results, and recalling that $\olp(2)=0$, we get
\begin{eqnarray*}
IG^{(3)}_0 (CT(X)) &\cong &  0\\
IG^{(2)}_1 (CT(X)) &\cong & \R^2\\
IG^{(1)}_2 (CT(X)) &\cong & \R^4\\
IG^{(0)}_3 (CT(X)) &\cong & \R^2\\
IG^{(-1)}_4 (CT(X)) &\cong & 0 .
\end{eqnarray*}
Thus
\[
\widetilde{HI}_j^\olp(X) \cong IG^{(3-j)}_j (CT(X)),
\]
where we see $3 = n - 1 + \olp(2)$,  as in Theorem \ref{ihhihom}.

\section{De Rham Cohomology for $IH$ and $HI$}

\subsection{Extended Perversities and the de Rham Complex for $IH$}

Let $W$ be a pseudomanifold with one
connected smooth singular stratum $B \subset W$ of codimension
$c$ and with link $F$ of dimension $f=c-1$.  (In what follows, we will take $W=CT(X)$, so $B=L$ and $F=\Sigma$.)  Then the only part of the perversity which affects $I\!H_\olp^*(W)$, is the value $\olp(c)$.  
Thus in this special case, we can simplify notation by labelling the intersection cohomology
groups by a number $p$ that depends only on the value $\olp(c)$, rather than by the whole function $\olp$.  
Further, we will fix notation such that the Poincar\'e lemma for a cone has the form:
\begin{equation}\label{ihpoincare}
IH^j_{(q)}(c^\circ F) = \left\{ 
\begin{array}{ll}
H^j(F) & j< q \\
0 & j \geq q.
\end{array} \right.
\end{equation}
That is, the $q$ we use in the notation $IH^j_{(q)} (W)$ gives the cutoff degree in the local cohomology calculation on the link.
The de Rham theorem for intersection cohomology \cite{BHS} states that in this situation,
\[
IH^*_{(c-1-\olp(c))}(W) \cong \Hom(IH^\olp_*(W), \real).
\]

Standard perversities satisfy $0 \leq \olp(c) \leq c-2$, so in terms of the convention we have introduced, this gives $0 < q \leq c-1$. We use an extension of these definitions
in which $q \in \mathbb{R}$.   This does not
give anything dramatically new; when $q \leq 0$,
we get $H^*(\olM,\partial \olM)$, 
where $\olM \cong W - N$ for $N$ an open tubular neighborhood of $B$.
When
$q > c-1$ we get
$H^*(\olM)$. It is also worth recording that when $W$ has an isolated conical singularity $B=\pt$ with 
link $F$, we get the following isomorphisms globally:
\[
IH^j_{(q)}(W) = \left\{ 
\begin{array}{ll}
H^j(\olM) & j< q \\
{\rm Im}(H^j (\olM, \partial \olM) \to H^j(\olM) ) & j=q \\
H^j (\olM,\partial \olM) & j > q,
\end{array} \right.
\]
where in this case $\partial\olM \cong F$.

In order to prove Theorem \ref{ihhi} from the de Rham perspective, we need to use compatible de Rham complexes
to define these cohomologies.  Various complexes have been shown to calculate intersection cohomology of a pseudomanifold.  
We will present first a version of the de Rham complex from \cite{BHS}, adapted to our setting.  

We use the notation from Section \ref{sec.notation}, and in particular let $W=CT(X)$ have an $l$-dimensional smooth singular stratum $L$ with 
link $\Sigma$, a smooth $s$-dimensional manifold, and product link bundle $Y\cong L \times \Sigma$.  Note that $s+1= \codim L$.   
From the isomorphism $Y \cong L \times \Sigma$, we have that $T^*Y \cong T^*L \oplus T^*\Sigma$.  This induces a 
bundle splitting
\[
\Lambda^k (T^* Y) \cong \bigoplus_{i+j=k} \Lambda^i (T^* L) \otimes \Lambda^j (T^* \Sigma).
\]
We write $\Omega^k (Y)=\Gamma^\infty (Y; \Lambda^k (T^* Y))$ for the space of smooth differential
$k$-forms on $Y$, $\Lambda^{i,j} Y = \Lambda^i (T^* L) \otimes \Lambda^j (T^* \Sigma)$ and
$\Omega^{i,j} (Y) = \Gamma^\infty (Y; \Lambda^{i,j} Y)$.
Then also we obtain a splitting of the space of smooth sections,
\[
\Omega^k (Y) \cong \bigoplus_{i+j=k} \Omega^{i,j}(Y), \qquad \alpha = \sum_{i,j} \alpha_{i,j}
\]
as $C^\infty(Y)$-modules.  For $q \in \mathbb{Z}$, define the fiberwise (along $\Sigma$) truncated space of forms over $Y$:
\[
(\ft_{<q}\Omega^* (Y))^k := \{ \alpha \in \Omega^k(Y) \mid \alpha = \sum_{j=0}^{q-1} \alpha_{k-j,j}\}.
\]
Note that $\ft_{<q}\Omega^* (Y)$ is not a complex in general.
Then define the complex:
\begin{equation}\label{IHcomplex}
I\Omega^*_{(q)}(CT(X)) := \{ \omega \in \Omega^*(M) \mid \omega = \eta \mid_M, \eta \in \Omega^*(\olM), \hspace{1in}
\end{equation}
\[
\hspace{2in} \inc^*\eta \in \ft_{<q}\Omega^*(Y), \inc^*(d\eta) \in \ft_{<q}\Omega^*(Y)\}.
\]
(Recall that $\inc:Y=\partial \olM \hookrightarrow \olM$ is the inclusion of the boundary.)
The cohomology of this complex is $IH_{(q)}^*(CT(X))$, as shown in, e.g.,  \cite{BHS}.
We note that there is an inclusion of complexes, 
\[
S_{q, q+1}: I\Omega^*_{(q)}(CT(X)) \hookrightarrow I\Omega^*_{(q+1)}(CT(X)).
\]
This induces a natural map on cohomology, which however is generally neither injective nor surjective.
We will come back to this map later.

\subsection{Extended Perversities and the de Rham Complex for $HI$}
\label{ssec.derhamcplxhi}

In this subsection, we present the de Rham complex defined in \cite{Ba-new}, which computes
the reduced singular cohomology of intersection spaces. Let $L$ be oriented and equipped with a
Riemannian metric.
For flat link bundles $E\to \Sigma$ whose link can be given a Riemannian metric such that the transition functions
are isometries, the first author defined in \cite{Ba-new} a subcomplex
$\Omega^*_{\mathcal{MS}}(\Sigma) \subset \Omega^* (E)$, the
complex of \emph{multiplicatively structured forms}. In the present special case of $E=Y=L\times \Sigma,$
this subcomplex is
\[
\Omega^*_{\mathcal{MS}}(\Sigma)= \{\omega \in \Omega^* (Y) \mid \omega = \sum_{i,j} \pi_L^*\lambda_i \wedge \pi_\Sigma^* \sigma_j\},
\]
where the sum here is finite, and the $\lambda_i \in \Omega^*(L)$ and $\sigma_j \in \Omega^*(\Sigma)$.  We 
may also write this as
\[
\Omega^*_{\mathcal{MS}}(\Sigma)\cong \Omega^*(L)\otimes \Omega^*(\Sigma).
\]
Let $k$ be any integer. The level-$k$ co-truncation of the complex $\Omega^* (L)$ is defined in \emph{loc. cit.} as
the subcomplex $\tau_{\geq k}\Omega^* (L) \subset \Omega^* (L)$ given in degree $m$ by
\[
(\tau_{\geq k}\Omega^* (L))^m =\left\{
\begin{array}{ll}
0 & m<k \\
\ker \delta_L & m=k \\
\Omega^m(L) & m>k,
\end{array} \right.
\]
using the codifferential $\delta_L$ on $\Omega^* (L)$.
Note that for $k\leq 0$, $\tau_{\geq k}\Omega^* (L) = \Omega^* (L)$, while for
$k>l,$ $\tau_{\geq k}\Omega^* (L) = 0$.
The subcomplex 
$\ft_{\geq k}\Omega^*_{\mathcal{MS}}(\Sigma) \subset \Omega^*_{\mathcal{MS}}(\Sigma)$ 
of \emph{fiberwise (along $L$) co-truncated forms} 
is given in degree $m$ by 
\[
(\ft_{\geq k} \Omega^*_{\mathcal{MS}}(\Sigma))^m = 
\{\omega \in \Omega^m(Y) \mid \omega = \sum_{i,j} \pi_L^*\lambda_i \wedge \pi_\Sigma^* \sigma_j, 
 \lambda_i \in \tau_{\geq k}\Omega^*(L)\}.
\]
Taking $k=l-\olp (l+1),$
we set
\[
HI^*_{\dR, \olp}(X):= H^*(\Omega I^*_{\overline{p}}(M)), 
\]
where
\[
\Omega I^*_{\overline{p}}(M):=\{
\omega \in \Omega^*(M) \mid \omega|_{N- \Sigma} = \pi_Y^*\eta, \, \eta \in \ft_{\geq l-\overline{p}(l+1)}\Omega^*_{\mathcal{MS}}(\Sigma)\}.
\]
(Recall from Section \ref{sec.notation} that $N$ is an open tubular neighborhood of
$\Sigma$ with a fixed diffeomorphism $N-\Sigma \cong L\times \Sigma \times (0,1)$
and $\pi_Y: N-\Sigma \to Y=L\times \Sigma$ is the projection.)
The Poincar\'e duality theorem of \cite{Ba-new} asserts that if $\olp$ and $\olq$
are complementary perversities, then wedge product of forms followed by integration
induces a nondegenerate bilinear form
\[ HI^*_{\dR, \olp}(X)\times HI^{n-*}_{\dR, \olq}(X)\to \real,~
 ([\omega],[\eta])\mapsto \int_{X-\Sigma} \omega \wedge \eta, \]
when $X$ is compact and oriented. (This is shown not just for trivial link bundles, but for
any flat link bundle whose transition functions are isometries of the link.)
Furthermore,
using a certain partial smoothing technique, the de Rham theorem of \cite{Ba-new}
asserts that for isolated singularities 
\begin{equation}  \label{equ.derhamforhi}
HI^*_{\dR, \olp} (X)\cong \Hom (\widetilde{H}_* (\ip X;\real),\real).
\end{equation}
This has been generalized by Essig in \cite{Ess} (Theorem 3.4.1) to nonisolated singularities with
trivial link bundle.

We can create a notation for $HI^*_{\dR, \olp}$ that emphasizes the cutoff degree instead of the perversity
in a similar vein to the notation we fixed for 
$IH_{\olp}^*(X)$ in the previous section. With $k= l-\overline{p}(l+1)$, we simply write
\[HI^*_{(k)}(X) := HI^*_{\dR, \overline{p}}(X).\]
Since the only value of $\olp$ to make a difference in the right side of this equation is $\olp (l+1)$, no ambiguity arises
from replacing the function $\olp$ by the number $k$, where now $k$ is giving the cutoff degree in 
the local calculation on the link.  

We observe two useful lemmas about the cohomology of the complex $\Omega I^*_\olp(M)$.  The first one is 
a generalised Mayer-Vietoris sequence.
\begin{lemma}\label{HIMV}
There is a long exact sequence of de Rham cohomology groups as follows:
\[
\cdots \to HI^j_{(k)}(X) \to H^j(\olM) \oplus H^j\left(\ft_{\geq k}\Omega^*_\mathcal{MS}(\Sigma)\right) 
\to H^j(\partial \olM) \to \cdots .
\]
In particular, since $\partial \olM = L \times \Sigma$, the second summand of the middle term is isomorphic to
\[
\bigoplus_{i=k}^l H^i(L) \otimes H^{j-i}(\Sigma).
\]
\end{lemma}
\begin{proof}  Let $k=l-\olp(l+1)$.
By definition of $\Omega I^*_{\overline{p}}(M)$, we have a short exact sequence of complexes:
\[
0 \to \Omega I^*_{\olp}(M) \to \Omega^*(M) \oplus \ft_{\geq k}\Omega_{\mathcal{MS}}^*(\Sigma) 
\to \Omega^*(\partial \olM) \to 0,
\]
where the second map takes a pair $(\omega, \eta)$ to $\omega| _{Y \times \{ 1/2 \}} - \eta$,
and the first map 
takes $\omega \in \Omega I^*_{\olp}(M)$ with 
$\omega|_{N- \Sigma} = \pi_Y^*\eta$ to $(\omega, \eta)$.
This sequence induces the long exact sequence on cohomology in the lemma.  The form of the second 
summand comes
from the definition of co-truncation and of multiplicatively structured forms.
\end{proof}

\begin{lemma}(K\"unneth for $HI_*$.) \label{lem.HIhomkun}
 If $W$ is a pseudomanifold with only one isolated singularity and $B$ is a closed manifold, then
the homological cross product induces an isomorphism
$\redh I^{\olp}_* (W\times B) \cong \redh I^{\olp}_* (W)\otimes H_* (B).$
\end{lemma}
\begin{proof}
Let $\olN$ be the blowup of $W$.
Set $k=l-\bar{p}(l+1),$ 
where $l$ is the dimension of the link $L=\partial \olN,$
and let $f:L_{<k} \to L$ be a stage-$k$ Moore approximation
to $L$.
Then $\ip W=\cone (g)$, where $g$ is the composition
\[ L_{<k} \stackrel{f}{\longrightarrow} L = \partial \olN \hookrightarrow \olN. \]
The blowup $\olM$ of $X=W\times B$ is $\olM = \olN \times B$ with boundary $\partial \olM = L\times B$.
The intersection space of $X$ is then $\ip X = \cone (g\times \id_B)$ because
$g\times \id_B$ is the composition
\[ L_{<k} \times B 
   \stackrel{f\times \id_B}{\longrightarrow} 
   L\times B = \partial \olM \hookrightarrow \olM=\olN \times B. \]
Let $v\in \cone (g)$ be the cone vertex and let $Q= (\cone (g)\times B)/(\{ v \} \times B)$,
which is homeomorphic to $\cone (g\times \id_B)$.
As the inclusion $\{ v \} \times B \to \cone (g)\times B$
is a closed cofibration, the quotient map induces an isomorphism
\[ H_* ((\cone (g), \{ v \})\times B)= 
  H_* (\cone (g)\times B, \{ v \} \times B) \cong \redh_* (Q) \cong 
  \redh_* (\cone (g\times \id_B)). \]
By the K\"unneth theorem for relative homology,
\[ H_* ((\cone (g), \{ v \})\times B) \cong
  H_* (\cone (g), \{ v \})\otimes H_* (B) = \redh_* (\cone (g))\otimes H_* (B). \]
Composing, we obtain an isomorphism
\[ \redh I^{\olp}_* (W\times B) =
  \redh_* (\cone (g\times \id_B)) \cong \redh_* (\cone (g))\otimes H_* (B)
  = \redh I^{\olp}_* (W)\otimes H_* (B) . \]
\end{proof}

\begin{lemma}(K\"unneth for $HI^*_{\dR}$.) \label{HIkun} 
If $W$ is a pseudomanifold with only one isolated singularity and $B$ is a smooth closed manifold, then
$HI^*_{\dR, \olp} (W \times B) \cong HI^*_{\dR, \olp} (W) \otimes H^*(B)$.
\end{lemma}
\begin{proof} 
We give two different proofs; the first one, however, assumes $W$ and $B$ to be compact.
In this case the homology groups 
$\redh I^\olp_* (W)$ and $H_* (B)$ are finite dimensional and thus
the natural map
\[ \Hom (\redh I^\olp_* (W),\real)\otimes \Hom (H_* (B),\real) \longrightarrow
 \Hom (\redh I^\olp_* (W) \otimes H_* (B), \real \otimes \real) \]
is an isomorphism. Thus, by Lemma \ref{lem.HIhomkun} and 
the de Rham isomorphism (\ref{equ.derhamforhi}),
\begin{align*}
HI^*_{\dR, \olp} (W\times B)
&\cong \Hom (\redh I^{\olp}_* (W\times B),\real) \\
&\cong \Hom (\redh I^{\olp}_* (W)\otimes H_* (B), \real) \\
& \cong \Hom (\redh I^{\olp}_* (W),\real)\otimes \Hom (H_* (B),\real) \\
& \cong HI^*_{\dR, \olp} (W) \otimes H^*(B).
\end{align*} 

The second argument does not require the compactness assumption.
Let $\olN$ be the blowup of $W$. 
Then if we put $X=W \times B$ in the long exact sequence from Lemma \ref{HIMV}, we get $\Sigma=B$, $\olM = \olN \times B$ and $\partial \olM = (\partial \olN) \times B$.  Thus the second two 
terms decompose as a tensor product with $H^*(B)$, so by the five lemma, so does the first term.
\end{proof}

These lemmas allow us to compute $HI^*_{(k)} (X)$ for values of extended 
perversities that lie outside of the topologically invariant range of Goresky-MacPherson.
For standard perversities, $1 \leq k \leq l$.  If $k \leq 0$, then 
$\tau_{\geq k}\Omega^* (L) = \Omega^* (L)$ and thus the sequence of Lemma \ref{HIMV} becomes
\[
\cdots \to HI^j_{(k)}(X) \stackrel{\phi}{\longrightarrow} H^j(\olM) \oplus H^j\left( \Omega^*_\mathcal{MS}(\Sigma)\right) 
\stackrel{\psi}{\longrightarrow} H^j(\partial \olM) \to \cdots .
\]
The map $\psi$ has the form $\psi = \psi_M - \psi_Y,$ where 
$\psi_M: H^j (\olM)\to H^j (\partial \olM)$ is restriction and 
$\psi_Y: H^j \left( \Omega^*_\mathcal{MS}(\Sigma)\right) \to H^j (\partial \olM)$ is induced by the
inclusion of complexes.
Since $\psi_Y$ is in the present case an isomorphism, the map $\psi$ is surjective and thus $\phi$ is injective.
Now
\[ \im \phi = \ker \psi = \{ (\omega, \eta) ~|~ \psi_M (\omega)=\psi_Y (\eta) \} =
 \{ (\omega, \psi^{-1}_Y \psi_M (\omega)) \}, \]
which is isomorphic to $H^* (\olM)$.
We conclude that
$HI^*_{(k)}(X) \cong H^*(\olM)$ when $k\leq 0$.
On the other hand, if $k >l$, then $\tau_{\geq k}\Omega^* (L) = 0$ and hence the 
sequence of Lemma \ref{HIMV} becomes
\[
\cdots \to HI^j_{(k)}(X) \longrightarrow H^j(\olM)  
 \longrightarrow H^j(\partial \olM) \to \cdots .
\]
Therefore, $HI^*_{(k)}(X) \cong H^*(\olM,\del \olM)$ when $k>l$.
Combining these observations with the homological Proposition \ref{prop.hiextcalc},
we see that the de Rham isomorphism
to $\Hom (\widetilde{H}_* (I^\olp X),\real)$ 
proved by Essig for standard perversities also holds
for extended perversities: If $k>l$, then
\[ HI^*_{(k)} (X)\cong H^* (\olM, \partial \olM) \cong
  \Hom (H_* (\olM, \partial \olM),\real) \cong 
  \Hom (\redh I^{\olp}_* (X),\real), \]
whereas if $k\leq 0$, then
\[ HI^*_{(k)} (X)\cong H^* (\olM) \cong
  \Hom (H_* (\olM),\real) \cong 
  \Hom (\redh I^{\olp}_* (X),\real). \]
Furthermore, Poincar\'e duality also then works for these extended perversities since relative and absolute (co)homology pair nondegenerately under the standard
intersection pairing.

\subsection{An Alternative de Rham Complex for $HI$}

We continue to assume that $L$ is oriented.
We now want to define a new, equivalent de Rham complex for $HI^*_{(k)}(X)$ that is analogous to the de Rham complex we presented above for $IH^*_{(q)}(CT(X))$.  In order to do this, we need to extend 
the operator $\delta_L$ from multiplicatively structured forms on $Y$ to all smooth forms on $Y$.  This is standard, but
we give details here for clarity.  First, we can decompose the exterior derivative according to the splitting of $Y$ as
\[
d_Y = \dL + (-1)^i \dS
\] 
for $(i,*)$-forms, where for 
${\bf z}=(z_1, \ldots,z_l)$ local coordinates on a coordinate patch $U\subset L$ and ${\bf y}=(y_1, \ldots, y_s)$
local coordinates on a coordinate patch $V \subset \Sigma$ and multi-indices ${\bf I}$ and ${\bf J}$, 
\[
\dL(f({\bf z}, {\bf y})dz_{\bf I}\wedge dy_{\bf J}):= \sum_{i=1}^l \frac{\del f}{\del z_i} dz_i \wedge dz_{\bf I}\wedge dy_{\bf J} 
\]
and
\[
\dS(f({\bf z}, {\bf y})dz_{\bf I}\wedge dy_{\bf J}):= \sum_{j=1}^s \frac{\del f}{\del y_j} dz_{\bf I}\wedge dy_j \wedge dy_{\bf J}.
\]
Note that $\dL(\pi_L^* \lambda\wedge \pi_\Sigma^*\sigma) =  \pi_L^*(d_L \lambda)\wedge \pi_\Sigma^*\sigma$ and $\dS(\pi_L^* \lambda\wedge \pi_\Sigma^* \sigma) = \pi_L^* \lambda\wedge \pi_\Sigma^*(d_\Sigma \sigma)$,
so these operators extend the exterior derivatives on multiplicatively structured forms to operators over all smooth forms
on $Y$.

Now fix a metric $g_L$ on $L$.  This defines a Hodge star operator on $\Omega^*(L)$, which may be extended to 
forms on $Y$ via the rule 
\[
\starL(dz_I \wedge dy_J):= (*_L dz_I) \wedge dy_J.
\]
Now we can extend the adjoint operator of $d_L$ to forms on $Y$ by setting for $(i,j)$-forms that
\[
\delL:=(-1)^{li +l+ 1} \starL\dL \starL.
\]
Note that this does extend the adjoint operator from multiplicatively structured forms.  From the coordinate 
definitions and the invariance of $g_L$ in the $V$ coordinates, we can observe:
\begin{equation}\label{useful}
\dL \dS = \dS \dL, \qquad
\dS \starL = \starL \dS,
\qquad \dS \delL = \delL \dS.
\end{equation}
Furthermore, we can lift the Hodge decomposition for $\Omega^*(L)$ to any neighborhood $L \times V$ by observing
that for any fixed ${\bf y} \in V$, we have a decomposition of $\omega \in \Omega^{i,j}(L \times V)$ given by
\[
\omega = \sum_{|J|=j}  \lambda_J({\bf y}) \wedge dy_J,
\]
where each $\lambda_J({\bf y}) \in \Omega^i(L)$ decomposes as $d_L\lambda_{1,J}({\bf y}) + \delta_L \lambda_{2,J}({\bf y}) + \lambda_{3,J}({\bf y})$,
with $\lambda_{3,J}({\bf y}) \in \calH^i(L)$.  Thus altogether we can decompose
\begin{eqnarray}\nonumber
\omega &=& \sum_{|J|=j} d_L\lambda_{1,J}({\bf y}) \wedge dy_J + \sum_{|J|=j} \delta_L\lambda_{2,J}({\bf y}) \wedge dy_J+\sum_{|J|=j}  \lambda_{3,J}({\bf y}) \wedge dy_J\\ \nonumber
&=& \dL \sum_{|J|=j} \lambda_{1,J}({\bf y}) \wedge dy_J + \delL\sum_{|J|=j} \lambda_{2,J}({\bf y}) \wedge dy_J+\sum_{|J|=j}  \lambda_{3,J}({\bf y}) \wedge dy_J\\
&=& \dL \omega_1({\bf y}) + \delL\omega_2({\bf y}) + \omega_3({\bf y}).
\end{eqnarray}
Now putting this together for all $y\in V$, we get a unique decomposition of any form in $\Omega^{i,j}(L\times V)$ into 
pieces in the image of $\dL$, in the image of $\delL$ and in the kernel of both.

\begin{lemma}\label{hodgelem}
We have the following decomposition, where the sums are vector space direct sums:
\[
\Omega^{i,j}Y = \dL \Omega^{i-1,j}(Y) \oplus \delL \Omega^{i+1,j}(Y) \oplus \left(\calH^i(L)\otimes\Omega^j(\Sigma)\right).
\]
Further, $\dS$ preserves this decomposition.
\end{lemma}
\begin{proof}
We have already demonstrated the decomposition, since this is done pointwise in $B$ (finite dimensionality of 
$\calH^i(L)$ allows us to write the last term of the decomposition as a tensor product).  The fact that it is preserved
by $\dS$ follows from (\ref{useful}).
\end{proof}
Note that since $L \times \Sigma$ is a product, we can also apply Lemma \ref{hodgelem} in the other direction,
namely that $\dL$ preserves the Hodge decomposition for $\Sigma$.  In this way, we get in fact a double
Hodge decomposition.
A graded vector space $\hft_{\geq k} \Omega^* (Y)$ of alternatively fiberwise co-truncated forms is given by
\[
(\hft_{\geq k}\Omega^* (Y))^j := \{ \alpha= \sum_{i=k}^l \alpha_{i, j-i} \mid \delL \alpha_{k,j-k}=0 \}.
\]
\begin{lemma}
The differential $d_Y$ restricts to $\hft_{\geq k} \Omega^* (Y)$.
\end{lemma}
\begin{proof}
The differential $d_Y = \dL \pm \dS$ does not lower the $L$-degree $i$ of a form $\alpha_{i,j-i}$.
Thus, if $\alpha= \sum_{i=k}^l \alpha_{i, j-i},$ then
$d_Y \alpha$ can again be written in the form
$d_Y \alpha= \sum_{i=k}^l \beta_{i, j+1-i},$ $\beta_{i,j+1-i}\in \Omega^{i,j+1-i} (Y)$.
Assume that $\delL \alpha_{k,j-k}=0$. Since
\begin{align*}
d_Y \alpha &= (\dL \pm \dS)(\alpha_{k,j-k} + \alpha_{k+1,j-k-1} + \cdots)\\
&=  \dL \alpha_{k,j-k} \pm \dS \alpha_{k,j-k} + \dL \alpha_{k+1,j-k-1} \pm \dS \alpha_{k+1,j-k-1} + \cdots,
\end{align*}
the component in bidegree $(k,j+1-k)$ of $d_Y \alpha$ is
\[ (d_Y \alpha)_{k,j+1-k} = \pm \dS \alpha_{k,j-k}. \]
Using (\ref{useful}),
\[ \delL (d_Y \alpha)_{k,j+1-k}  = \pm \delL \dS \alpha_{k,j-k} = \pm \dS \delL \alpha_{k,j-k} =0. \]
This shows that $d_Y \alpha \in (\hft_{\geq k}\Omega^* (Y))^{j+1}.$
\end{proof}
By the lemma, $\hft_{\geq k} \Omega^* (Y)$ is a differential complex.
Now we can define the new deRham complex for $HI^*_{(k)}(X)$ as:
\begin{equation}\label{dRHI}
\widehat{\Omega I}^*_{(k)}(X):=\{ \mu|_M \,\mid  \, \mu \in \Omega^*(\olM), \, \inc^*(\mu) \in 
\hft_{\geq k}\Omega^*(Y)\}.
\end{equation}
We want to show this complex is equivalent to the original de Rham complex.  For this, we will need 
the K\"unneth Theorem.
\begin{thm}(K\"unneth Theorem) \label{kunneth}  Let $Y=L \times \Sigma$.  Then the inclusion of complexes
\[
\Omega^*(L) \otimes \Omega^*(\Sigma) \cong \Omega_\mathcal{MS}(\Sigma) \hookrightarrow \Omega^*(Y)
\]
induces an isomorphism on cohomology.  In particular, if $\alpha \in \Omega^*_\mathcal{MS}(\Sigma)$ and 
$\alpha = d\beta$ for $\beta \in \Omega^*(Y)$, then $\beta = d\gamma + \beta'$, 
where $\beta' \in \Omega^*_\mathcal{MS}(\Sigma)$.
\end{thm}
\begin{lemma}\label{hieq}
The cohomology, $\widehat{HI}^*_{(k)}(X)$, of the complex $\widehat{\Omega I}^*_{(k)}(X)$ is isomorphic to $HI^*_{(k)}(X)$.
\end{lemma}
\begin{proof}
We first note that there is an inclusion of complexes
\[
{\Omega I}^*_{(k)}(X) \hookrightarrow \widehat{\Omega I}^*_{(k)}(X),
\]
which thus induces a map on cohomology.  We need to show this map is a bijection.
We start with injectivity.  Assume that $[\alpha] \in HI^j_{(k)}(X)$ and that $\alpha = d\hat{\beta}$ for 
$\hat{\beta} \in \widehat{\Omega I}^{j-1}_{(k)}(X)$.  Since the two complexes differ only by the structure of forms
on $N-\Sigma$, we restrict our consideration to this neighborhood.
Here we have 
\[\hat{\beta} = \hat{\beta}_t(x) + dx \wedge \hat{\beta}_n(x),
\]
where $x$ is the cone coordinate in $[0,1)$.  Let 
\[
\beta=\hat{\beta} - d\chi \int_0^x \hat{\beta}_n(t)\, dt,
\]
where $\chi$ is a smooth cutoff function on $M$ which is identically $=1$ on $N-\Sigma$.  Note that 
$\inc^*(\beta)=\inc^*(\hat{\beta})$, so $\beta \in \widehat{\Omega I}^{j-1}_{(k)}(X)$.  But additionally, 
$dx \lrcorner \, \beta =0$; that is, $\beta=\hat{\beta}_t(x) + dx \wedge 0$.  
Now from the fact that $\alpha = \pi_Y^*\eta$ for some $\eta \in \Omega^j_{\mathcal{MS}}(\Sigma)$
and $d\beta= d\hat{\beta}=\alpha$, we have on $N -\Sigma$ that
\[
\pi_Y^*\eta = d\beta(x) = d_Y\beta(x) + dx \wedge \beta'(x).
\]
Thus we must have $\beta'(x)=0$; that is, $\beta=\pi_Y^*\sigma$ for some $\sigma \in
(\hft_{\geq k}\Omega^* (Y))^{j-1}$.  Since $d$ commutes with pullbacks, we have $d_Y\sigma = \eta$.  
We need to show that we can further adapt $\hat{\beta}$ to a $\tilde{\beta} = \pi_Y^*\tilde{\sigma}$
where $\tilde{\sigma}$ is multiplicatively structured and in the right co-truncated complex.  When
we have this, we have completed the injectivity proof.

The injectivity of the K\"unneth isomorphism, $d_Y\sigma = \eta$, where $\eta$ is multiplicatively
structured, implies that $\sigma = d_Y \gamma + \sigma'$, where $\sigma'$ is multiplicatively 
structured.  Breaking this equation down by bidegree, we have
\begin{equation}\label{sigma1}
\sigma = d_Y \left(\gamma_{j-2,0}+ \cdots + \gamma_{k, j-2-k}\right) + \dL \gamma_{k-1,j-1-k}
+ \sigma'_{j-1,0}+ \cdots + \sigma'_{k, j-1-k}
\end{equation}
and
\[
0 = (-1)^{k-1} \dS \gamma_{k-1,j-1-k} + d_Y\left(\gamma_{k-2,j-k}+ \cdots + \gamma_{0,j-2}\right) \qquad \qquad
\]
\[\qquad \qquad+ \sigma'_{k-1, j-k} + \cdots + \sigma'_{0,j-1}.
\]
In particular, in Equation \ref{sigma1}, in bidegree $(k,j-1-k)$, we know that $\delL \sigma_{k,j-1-k}=0.$
This means that this equation will still hold if we eliminate $\gamma_{k-1,j-k-1}$ and replace 
$\sigma'_{k,j-1-k}$ with its $\delL$ + harmonic in $L$ components from the $L$ Hodge decomposition.  
Then we can additionally assume all of the terms in the second equation are 0.  This means 
we have written $\sigma = d_Y\gamma + \sigma'$, where 
$\sigma' \in (\ft_{\geq k}\Omega^*_{\mathcal{MS}}(\Sigma))^{j-1}$ and 
$\gamma \in (\hft_{\geq k}\Omega^* (Y))^{j-2}$.  Now 
\[
\tilde{\beta} := \beta - d\chi\pi_Y^*\gamma \in \Omega I^{j-1}_{(k)}(X)
\]
and $d\tilde{\beta}=\alpha$, so $0=[\alpha]\in HI^*_{(k)}(X)$, thus the map is injective.

\smallskip
Now consider surjectivity.  Let $[\hat{\beta}]\in \widehat{HI}^j_{(k)}(X)$.  Then on $N-\Sigma$, we again have
\[
\hat{\beta} = \hat{\beta}_t(x) + dx \wedge \hat{\beta}_n(x).
\]
As before, define 
\[
\beta=\hat{\beta} - d\chi \int_0^x \hat{\beta}_n(t)\, dt.
\]
Now by the same argument as above, $[\hat{\beta}] = [\beta]$, and on $N-\Sigma$, $\beta = \pi_Y^*\sigma$.
Further, $d\hat{\beta}=0$ implies that $d_Y \sigma =0$.  Using the Hodge decomposition on $Y$, we can write
$\sigma = d_Y \eta + \sigma'$, where $\sigma'$ is harmonic, and therefore multiplicatively
structured.  Now decomposing by bidegree and using orthogonality of the Hodge decomposition and 
the fact that $\sigma \in \hft_{\geq k}\Omega^{j}(Y)$, 
we get that $\sigma' \in (\ft_{\geq k}\Omega^*_{\mathcal{MS}}(\Sigma))^j$.  We also get that
$d_Y(\eta_{k-1,j-k} + \cdots + \eta_{0,j-1})=0$.  Thus we can without loss of generality assume these
terms are zero.  Finally, we have that $\delL\dS \eta_{k, j-k-1}= \delL (\sigma_{k, j-k} - \sigma'_{k, j-k})=0$.
Thus $\eta \in (\hft_{\geq k}\Omega^* (Y))^{j-1}$.  So let $\tilde{\beta} = \beta-d\chi\eta$.  Then
$[\tilde{\beta}] = [\hat{\beta}] \in \widehat{HI}^j_{(k)}(X)$ and $[\tilde{\beta}]$ is a class in 
$HI^j_{(k)}(X)$, so the map is surjective.
\end{proof}

\subsection{Proof of Theorem \ref{ihhi}}

Theorem \ref{ihhi} follows from a sequence of lemmas relating the spaces $IH_{(q)}^j(CT(X))$, $HI^j_{(j-q)}(X)$ and $IH_{(q+1)}^j(CT(X))$.
First we have the following lemma, which shows that there is a sequence of maps in each degree $j$:
\begin{lemma} For all $j$, there are well defined maps
\[
IH_{(q)}^j(CT(X)) \stackrel{A}{\longrightarrow} HI^j_{(j-q)}(X)\stackrel{B}{\longrightarrow} IH_{(q+1)}^j(CT(X))
\]
that factorise the standard map $S_{q,q+1}$ between intersection cohomology groups of adjacent perversities.
\end{lemma}
\begin{proof}
First consider the map $A$.  Let $\alpha \in I\Omega_{(q)}^j(CT(X))$, $d\alpha=0$.  
Then by definition of this complex, $\inc^*\alpha \in (\ft_{<q} \Omega^* (Y))^j$.
We can decompose $\inc^*\alpha$ by $(L,\Sigma)$ bidegree to get
\[
\inc^* \alpha = \alpha_{j,0} + \cdots + \alpha_{j-q+1,q-1}.
\]
In particular, $\alpha_{j-q,q}=0$, which means that $\alpha \in  \widehat{\Omega I}^j_{(j-q)}(X)$.  
To show that this inclusion induces a map on cohomology, we need to know that if $\alpha= d\eta$ where 
$\eta \in I\Omega_{(q)}^{j-1}(CT(X))$, then we can find $\eta' \in \widehat{\Omega I}^{j-1}_{(j-q)}(X)$ 
so that $\alpha = d\eta'$, as well.  
Decomposing by bidegree, we get
\[
\inc^* \eta = \eta_{j-1,0} + \cdots + \eta_{j-q,q-1}, \qquad d_\Sigma \eta_{j-q,q-1}=0.
\]
Decompose $\eta_{j-q,q-1}$ by the $L$ Hodge decomposition: 
\[
\eta_{j-q,q-1} = \dL a_{j-q-1,q-1} + \delL a'_{j-q+1,q-1} + h_{j-1,q-1}.
\]
Because the $\Sigma$ Hodge decomposition commutes with this decomposition, we have that 
$\dS a_{j-q-1,q-1}=0$.  
So let $\eta' = \eta - d\chi a_{j-q-1,q-1}$, where $\chi$ is a smooth cutoff function supported on the end.  Then of course 
$d\eta'=\alpha$ still, and 
\begin{eqnarray*}
\inc^* \eta' &=& \inc^* \eta - d_Y a_{j-q-1,q-1}\\
&=& \eta_{j,0} + \cdots + (\eta_{j-q,q-1} - \dL a_{j-q-1,q-1}) \pm \dS a_{j-q-1,q-1}\\
&=& \eta_{j,0} + \cdots + (\eta_{j-q,q-1} - \dL a_{j-q-1,q-1}),
\end{eqnarray*}
where by construction, $\delL (\eta_{j-q,q-1} - d_L a_{j-q-1,q-1})=0$, so 
$\eta' \in \widehat{\Omega I}^{j-1}_{(j-q)}(X)$,
and the map $A$ is well-defined.

\medskip
Now consider the map $B$.  Suppose that $\beta \in \widehat{\Omega I}^{j}_{(j-q)}(X)$ and $d\beta =0$.  Then 
\begin{equation}\label{beta}
\inc^*\beta = \beta_{j,0} + \cdots + \beta_{j-q, q}, \qquad \delL \beta_{j-q,q}=0,
\end{equation}
and so decomposing $\inc^*d\beta$ by bidegree, we have:
\begin{eqnarray*}
0 &=& \inc^*d\beta = d_Y \inc^* \beta \\
&=& (\dL \beta_{j,0}) + (\dL \beta_{j-1,1} + (-1)^j \dS \beta_{j,0}) 
+ \cdots \pm (\dS \beta_{j-q,q}).
\end{eqnarray*}
Thus $\dS \beta_{j-q,q}=0$, so $\beta \in I\Omega_{(q+1)}^j(CT(X))$.  Now we need to show the map induced by inclusion
is well defined on cohomology. Assume $\beta = d\mu$
for $\mu \in \widehat{\Omega I}^{j-1}_{(j-q)}(X)$; then decomposing by bidegree again, we have by definition of
$\widehat{\Omega I}^{j-1}_{(j-q)}(X)$ that
\[
\inc^* \mu = \mu_{j-1,0}+ \cdots \mu_{j-q, q-1}, \qquad \delL \mu_{j-q,q-1}=0.
\]
In particular, $\mu_{j-q-1,q}=0$, so $\mu \in I\Omega_{(q+1)}^{j-1}(CT(X))$, so $B$ is well-defined.

\medskip
Finally, since on the form level, $A$ and $B$ are both given by inclusion of a closed form in the domain
complex into the range complex, their composition factorises the natural map $S_{q,q+1}$.
\end{proof}

Next we have three lemmas that show $A$ is injective, $B$ is surjective and Kernel($B) \subset$ Image($A$).  Together, 
these prove Theorem \ref{ihhi}.

\begin{lemma} The map $A$ is injective.
\end{lemma}
\begin{proof}
Suppose that $A[\alpha]=0$, that is, $\alpha \in I\Omega_{(q)}^j(CT(X))$, $d\alpha=0$ and $\alpha = d\beta$ for 
$\beta \in \widehat{\Omega I}^{j-1}_{(j-q)}(X)$.  Then decomposing by bidegree,
\[
\inc^*\beta = \beta_{j-1,0} + \cdots + \beta_{j-q, q-1}, \qquad \delL \beta_{j-q,q-1}=0.
\]
Because the degree in $\Sigma$ is $\leq q-1$ for all pieces, $\inc^*\beta \in (\ft_{<q}\Omega^* (Y))^{j-1}$.
Also, by hypothesis, $d\beta = \alpha$ where $\inc^*\alpha \in (\ft_{<q}\Omega^* (Y))^j$, so 
$\beta \in I\Omega_{(q)}^{j-1}(CT(X))$, and $0=[\alpha] \in IH_{(q)}^j(CT(X))$.  Thus $A$ is injective.
\end{proof}

\begin{lemma} The map $B$ is surjective.
\end{lemma}
\begin{proof}
Suppose that $[\gamma] \in IH_{(q+1)}^j(CT(X))$.  Then decomposing by bidegree, we have
\[
\inc^*\gamma = \gamma_{j,0} + \cdots + \gamma_{j-q, q}.
\]
Since $d\gamma=0$, we get that $\dS \gamma_{j-q, q}=0$.  Decompose $\gamma_{j-q, q}$ according
to the $L$ Hodge decomposition: 
\[
\gamma_{j-q, q}= \dL a_{j-q-1,q} + \delL b_{j-q+1,q} + c_{j-q,q}.
\]
Using the double Hodge decomposition, we can assume $a_{j-q-1,q}$ is in the kernel of $\dS$.

Then let $\gamma' = \gamma - d \chi a_{j-q-1,q}$, where as before, $\chi$ is a smooth cutoff
function supported near the end.  Note that 
$\inc^*(\chi a_{j-q-1,q}) = a_{j-q-1,q} \in (\ft_{<q+1}\Omega^* (Y))^{j-1}$
and 
\[
\inc^*(d\chi a_{j-q-1,q})= d_Y a_{j-q-1,q} = \dL a_{j-q-1,q} \in (\ft_{<q+1}\Omega^* (Y))^j.
\]
Thus $[\gamma' ]=[\gamma] \in IH_{(q+1)}^j(CT(X))$.  Further, 
\[
\inc^* \gamma' = \gamma_{j,0} + \cdots + (\gamma_{j-q, q} - \dL a_{j-q-1,q}),
\]
so $\gamma' \in \widehat{\Omega I}^{j}_{(j-q)}(X)$, and $[\gamma] =B[\gamma']$, so $B$ is surjective.
\end{proof}

\begin{lemma} The kernel of $B$ is contained in the image of $A$.
\end{lemma}
\begin{proof}
Assume that $B[\beta]=0,$ that is, $\beta \in \widehat{\Omega I}^{j}_{(j-q)}(X)$ and $\beta = d\gamma$ for 
$\gamma \in I\Omega_{(q+1)}^{j-1}(CT(X))$.  Then decomposing $\inc^*\beta$ by bidegree as in Equation \ref{beta} and 
using the fact that $d\beta=0$,
we get that $\dS \beta_{k-q,q}=0= \delL \beta_{k-q,q}$.

Now decomposing $\inc^* \gamma$ and $\inc^*d\gamma$ by bidegree, we get that 
\[
\inc^* \gamma = \gamma_{j-1,0} + \cdots + \gamma_{j-1-q,q},
\]
\begin{eqnarray*}
\inc^* d\gamma &=& \left(\dL \gamma_{j-1,0}\right) + \left((-1)^{j-1}\dS \gamma_{j-1,0} + \dL \gamma_{j-2,1}\right)+ \cdots\\
&+& \left((-1)^{j-q}\dS \gamma_{j-q,q-1} + \dL \gamma_{j-1-q,q}\right) +\left((-1)^{j-q-1} \dS\gamma_{j-1-q,q}\right).
\end{eqnarray*}
Thus $\dS\gamma_{j-1-q,q}=0$ and 
\begin{equation}\label{gam}
(-1)^{j-q}\dS \gamma_{j-q,q-1} + \dL \gamma_{j-1-q,q}=\beta_{k-q,q}.
\end{equation}
Decompose $\gamma_{j-q,q-1}$ by the Hodge decomposition in $L$:
\[
\gamma_{j-q,q-1} = \dL a_{j-q-1,q-1} + \delL b_{j-q+1,q-1} + c_{j-q,q-1}.
\]
Then recalling that $\delL \beta_{k-q,q}=0$ and applying the Hodge decomposition in $L$ to
all of Equation \ref{gam}, we get
\[
\beta_{k-q,q} = (-1)^{j-q}\dS \left(\delL b_{j-q+1,q-1} + c_{j-q,q-1}\right).
\]
Let 
\[\beta' := \beta - d\chi (\delL b_{j-q+1,q-1} + c_{j-q,q-1}).
\]
Note that $(\delL b_{j-q+1,q-1} + c_{j-q,q-1}) \in \hft_{\geq j-q} \Omega^*(Y)$, so 
$[\beta'] = [\beta] \in HI^j_{(j-q)}(X)$.
But 
\begin{eqnarray*}
\inc^* \beta' &=& \inc^*\beta - d_Y(\delL b_{j-q+1,q-1} + c_{j-q,q-1}) \\
&=& \beta_{j,0} + \cdots + \beta_{j-q+1,q-1} - \dL \delL b_{j-q+1,q-1},
\end{eqnarray*}
so $\beta' \in I\Omega_{(q)}^{j}(CT(X))$.  Thus $[\beta] = A[\beta']$.
\end{proof}

\section{The Hodge Theorem for $HI$}
\label{sec.hodgehi}

Our Hodge theorem relates to the spaces of extended weighted $L^2$ harmonic forms over $M=X -\Sigma$ with respect to the various metrics we consider.  
A weighted $L^2$ space for any metric $g$ on $M$
is a space of forms:
\[
x^cL^2_{g}\Omega_{g}^*(M) := \{ \omega \in \Omega_{}^*(M) \mid
\int_M ||x^{-c} \omega||_{g}^2 {\rm dvol}_{g} < \infty\}.
\]
Here $||\,  \cdot \, ||_{g}$ is the pointwise metric on the space of differential forms over $M$
induced by the metric on $M$.
The space $x^cL^2_{g}\Omega_{g}^*(M)$ can be completed to a Hilbert space with respect to the inner product
\[
\langle\alpha, \beta\rangle_c := \int_M \alpha \wedge x^{-2c}*_{g} \beta.
\]

Let $d$ represent the de Rham differential 
on smooth forms over $M$ and $\delta_{g,c}$ represent its formal adjoint with respect to the $x^cL^2$ inner
product induced by the metric $g$.  Then $D_{g,c}:= d + \delta_{g,c}$ is an elliptic differential operator on 
the space of smooth forms over $M$.  If $c=0$, the elements of the kernel of $D_{g,0}$ that lie in $L^2$ are the standard 
space of $L^2$ harmonic forms over $(M,g)$.  More generally, we denote:
\[
\mathcal{H}^j_{L^2}(M, g,c):= \left\{\omega \in  x^cL^2_{g}\Omega^j_{g}(M) \mid D_{g,c} \omega =0 \right\}.
\] 
\begin{defn} \label{def.extforms}
The space of {\em extended} 
$x^cL^2$ harmonic forms on $(M,g)$ is
\[
\mathcal{H}_{ext}^*\left(M, g, c \right) := \bigcap_{\epsilon>0} \{ \omega \mid  \omega \in x^{c-\epsilon}L^2_{g} \Omega_{g}^*(M), \, D_{g,c} \omega=0\}.
\]
\end{defn}

\subsection{Proof of Theorem \ref{hodge}}
The space 
$IG_{(q)}^j(W)$ arises in extended Hodge theory for manifolds with fibred cusp metrics, and this allows us to
prove Theorem \ref{hodge}.
First, Theorem 1.2 from \cite{Hu3} may be rephrased as:
\begin{thm}\cite{Hu3}\label{ihthm}
Let $(M, g_{fc})$ be the interior of a manifold
with boundary $\overline{M}$ and boundary defining function $x$.  Assume that $\phi:\partial \overline{M} \to B$ is a 
fibre bundle that is flat with respect to the structure group ${\rm Isom}(F,g_F)$ for a fixed metric $g_F$ on the 
fibres of $\phi$.  Let $\hat{M}$ denote the compactification of $M$ obtained by collapsing the fibres of $\phi$ at $\partial\overline{M}$.
Endow $M$ with a geometrically flat fibred cusp-metric for the fibration $\phi$. 
Then 
\[
\mathcal{H}^j_{ext}(M, g_{fc},c) \cong IG_{((f/2)+1-c)}^j (\hat{M}),  \]
where $f=\dim F$.
\end{thm}

\begin{cor}\label{fbmetrics}
Under the conditions of Theorem \ref{ihthm}, if $g_{fb} = x^{-2}g_{fc}$ is the fibred boundary metric conformal to $g_{fc}$,
then 
\begin{equation}
\label{extforms}
\mathcal{H}_{ext}^{j}(M, g_{fb}, c) 
= \mathcal{H}_{ext}^{j}(M, g_{fc}, c+(n/2)-j) \cong
IG^j_{(q)}(\hat{M}),
\end{equation}
where $q=j-\frac{b-1}{2}-c$.
\end{cor}
\begin{proof}
If we take $g_{fb}$ to be the conformally related fibred boundary metric on $M$, then the conformal relationship 
$g_{fb} = x^{-2}g_{fc}$ means that
\[
x^c L^2_{fb}\Omega^j_{fb}(M) = x^{c+\frac{n}{2} - j}L^2_{fc}\Omega^j_{fc}(M).
\]
This means for the Hodge star operators that also  $*_{fb,c} = *_{fc,c+(n/2)-j}$, so in fact the extended 
harmonic forms in these spaces are the same.
\end{proof}

If the boundary fibration $\phi$ is a product, $\phi:F\times B \to B$, then we can also define the dual fibration,
$\psi:F \times B \to F$.  If $M$ is the interior of $\overline{M}$ where the fibration structure on $\partial\overline{M}$ is
given by $\phi$, then let $M'$ denote the same manifold, but where we now take the fibration structure
on $\partial\overline{M}$ to be given by $\psi$.  Then a fibred boundary metric on $M$ (i.e., with respect to the fibration $\phi$) is a fibred scattering
metric on $M'$ (i.e., with respect to $\psi$), and $\hat{M} = CT(\hat{M'})$.  Thus we can also write

\begin{cor}\label{cfhodge}
Under the conditions of Theorem \ref{ihthm}, if $(M',g_{fs}) = (M,g_{fb})$ is the 
fibred scattering metric on $M'$, that is,
$M$ where the boundary fibration is $\psi:F \times B \to F$,
then 
\begin{equation}
\label{extforms2}
\mathcal{H}_{ext}^{j}(M', g_{fs}, c) 
\cong
IG^j_{(q)}(\hat{M}) = IG^j_{(q)}(CT(\hat{M'})),
\end{equation}
where $q=j-\frac{b-1}{2}-c$.
\end{cor}

Theorem \ref{hodge} follows from this last corollary and Theorem \ref{ihhi} by taking $M' = X-\Sigma$ and $\hat{M'}=X$:
\begin{eqnarray*}
HI^j_{\dR, \olp} (X) & \cong & IG^j_{(j+1-l+\olp(l+1))}(CT(X)) \\
&\cong& \calH^j_{ext}(M,g_{fc}, \frac{n}{2}-j + \frac{l-1}{2}-\olp(l+1)) \\
&=& \calH^j_{ext}(M,g_{fb},\frac{l-1}{2}-\olp(l+1)) \\
&=& \calH^j_{ext}(M',g_{fs},\frac{l-1}{2}-\olp(l+1)).
\end{eqnarray*}

\subsection{Example}
We consider the same space $X=S^2 \times T^2$ as in Example \ref{example1}, stratified as before.  
Then $M\cong \mathbb{R} \times S^1 \times T^2$.  We can endow this with a geometrically flat fibred scattering metric:
\[
g_{fs} = dr^2 + (1+r^2)d\theta_1^2 + d\theta_2^2 + d\theta_3^2.
\]
Note that if we make the change of coordinates $x=|r|^{-1}$ near $\pm \infty$, we get a metric
that is a perturbation of one of the form in Definition \ref{cfmetric} that decays like $x^2$.  This
turns out to be sufficient to use the same analysis (see \cite{GH2}).  If we consider
extended $L^2$ harmonic forms on $(M, g_{fs})$ with no weight ($c=0$), then Theorem \ref{hodge}
says 
\[
\mathcal{H}^*_{ext}(M,g_{fs},0) \cong HI^*_{\dR,\olp} (X),
\]
where $(1-1)/2-\olp(2)= 0$.  That is, the spaces of extended unweighted $L^2$ harmonic forms on $M$ should be
isomorphic to the spaces with $\olp(2)=0$ as we calculated in Section \ref{example1}.

In order to identify the extended $L^2$ harmonic forms on $(M,g_{fs})$ it is useful 
to observe a few things.  First, since the metric is a global
product metric, the extended $L^2$ harmonic forms on $M$ are all products of extended $L^2$
harmonic forms on $W=\mathbb{R} \times S^1$
with harmonic forms on $T^2$.  Thus it suffices to determine the extended harmonic forms
on $W$ with the metric $g_W:=dr^2 + (1+r^2)d\theta_1^2$.  

Second, we observe that $g_W$ is a scattering metric, and is thus
conformally invariant (with conformal factor $(1+r^2)$) to a b-metric.  By the same argument 
as in Corollary \ref{fbmetrics}, this means that extended harmonic forms on $(W, g_W)$ are 
the same as extended weighted $L^2$ harmonic forms on $(W,(1+r^2)^{-1}g_W)$.  These 
forms are, in turn, known to be in the kernel of $d$ and $\delta$ independently (see either Proposition 
6.16 in \cite{meaps} or Lemma 4.3 in \cite{Hu3}).  Thus we know that extended harmonic $L^2$ forms
on $W$ are both closed and co-closed.  This means that the only possible 0-forms are constants
and the only possible 2-forms are constant multiples of the volume form.  

Third, recall that for a differential form to be extended harmonic, it must be in $x^{-\epsilon} L^2\Omega^*(W, g_W)$ for all $\epsilon>0$, or equivalently, $(1+r^2)^{-\epsilon}w \in L^2\Omega^*(W, g_W)$.  If we consider constant functions, 
this means we need 
\[
\int_{-\infty}^\infty c^2 (1+r^2)^{1-2\epsilon} \, dr <0.
\]
This is not true, so $\mathcal{H}^0_{ext}(W,g_W) = \{0\}$.  By an analogous argument (or equivalently, by Poincar\'e
duality), also $\mathcal{H}^2_{ext}(W,g_W) = \{0\}$.  

Finally consider $\mathcal{H}^1_{ext}(W,g_W)$.  The space of extended $L^2$ harmonic forms of middle
degree is preserved by a conformal change of metric, and as noted before, $g_W$ is conformally 
equivalent to the metric
\[
g_b = \frac{dr^2}{1+r^2} + d\theta_1^2.
\]
If we reparametrise, setting $t={\rm arcsinh}(r)$, this becomes the metric on the infinite cylinder:
\[
g_b = dt^2 + d\theta_1^2.
\]
If we use a Fourier series decomposition in $\theta_1$, we find that a 1-form
\[
\omega = \eta_0(r)\, dr + \sum_{n=0}^\infty \left(\eta_{1,n}(r)\cos(n\theta_1) + \eta_{2,n}(r)\sin(n\theta_1)\right) \, dr
\]
\[
+ \mu_0(r) \, d\theta + \sum_{n=0}^\infty \left(\mu_{1,n}(r)\cos(n\theta_1) + \mu_{2,n}(r)\sin(n\theta_1)\right) \, d\theta_1
\]
is closed and coclosed if $\eta_0(r)$ and $\mu_0(r)$ are constant
and the remaining coefficients satisfy $f'' = n^2 f$, that is, they are all exponential functions in $t$,
and thus blow up at either $\infty$ or $-\infty$, so are not almost in $L^2$.  So the
only extended $L^2$ harmonic forms are $c_1 d\theta_1 + c_2dt$, which
are in  $(1+t^2)^\epsilon L^2\Omega^*(W, g_W)$ as required.
Thus $\mathcal{H}^1_{ext}(W,g_W) \cong \mathbb{R}^2$.
Now when we take the tensor product with $\mathcal{H}^*(T^2)$, we get 
\begin{eqnarray*}
\mathcal{H}^0_{ext}(M,g_{fs}) \cong 0 \cong HI^0_{\dR, \olp} (X) \\
\mathcal{H}^1_{ext}(M,g_{fs}) \cong \mathbb{R}^2\cong HI^1_{\dR, \olp} (X)\\
\mathcal{H}^2_{ext}(M,g_{fs}) \cong \mathbb{R}^4\cong HI^2_{\dR, \olp} (X)\\
\mathcal{H}^3_{ext}(M,g_{fs}) \cong \mathbb{R}^2\cong HI^3_{\dR, \olp} (X)\\
\mathcal{H}^4_{ext}(M,g_{fs}) \cong 0\cong HI^4_{\dR, \olp} (X),
\end{eqnarray*}
as predicted by Theorem \ref{hodge}.

\subsection{Inclusion Map for the Hodge Theorem}
It is useful if we can understand the map from extended harmonic forms to $HI$ cohomology as 
given by sending an extended harmonic form to the $HI$ class that it represents:  $\gamma \to [\gamma]$, 
as in the classical Hodge theorem.  However, the extended harmonic forms in our Hodge theorem do not 
lie in either of the two complexes we have seen that calculate $HI^*_{\dR,\olp} (X)$.  To see them 
as representatives of classes, we need new spaces of forms that can be 
used to calculate the $HI$ cohomology spaces and that do contain the extended harmonic forms.  
We can find spaces that work in this regard by reinterpreting the proof of Theorem \ref{hodge}.  We
can find appropriate new spaces of forms by using the isomorphism with $IG$ and alternative
complexes of forms that may be used to calculate $IH$.  

From \cite{Hu3}, we have the following setup and lemma which will allow us to see the extended harmonic forms as representing classes in $HI$.
Assume that $W$ is a pseudomanifold with a single, smooth singular stratum, $B$, whose link bundle with link $F$ is flat with respect to the structure group ${\rm Isom}(F,g_F)$ for some fixed metric
on $F$.  Let $M=W-B$ and let $x$ be a smooth function on $M$ that extends across $B$ in $W$ by zero. 
Let $\olM$ be the complement in $W$ of a normal neighborhood of $B$, and 
let $i_s:\partial \olM \to M$ denote the inclusion of $\partial \olM$ into $M$ in the slice where $x=s$. 

Define the projection operator $\Pi_{q-1,q}$ on $\Omega^*(\partial\olM)$ by projection onto forms
in fibre degree $q$ that lie in ${\rm Ker}(\tilde{\delta}_F)$ and forms in fibre degree $q+1$ that lie in 
${\rm Image}(\tilde{d}_F)$ in terms of the $F$ Hodge decomposition.  
Let $x^aL^2\Omega^*_{con}(M,g_{fc})$ denote the complex of forms on $M$ that are conormal at $x=0$ (see,
e.g. \cite{meaps}), and are also in the $x^a$ weighted $L^2$ space on $M$ with respect to the metric $g_{fc}$.

\begin{lemma}\label{lemma6.3}  The cohomology of the complex:
$x^{(f/2)-q - \epsilon}L^2\Omega^*(M,g_{fc})$ (made into a complex in the standard way by requiring both $\omega$
and $d\omega$ to lie in the appropriate spaces) is isomorphic to $IH_{(q)}^*(W)$
and the cohomology of the complex:
\[
x^{(f/2)-q - \epsilon}L^2\Omega_{0}^*(M,g_{fc}):=\{ \omega \in x^{(f/2)-q - \epsilon}L^2\Omega_{con}^*(M,g_{fc}) \qquad \qquad\]
\[\hspace{2cm}
\mid \lim_{s \to 0} \Pi_{q-1,q}\circ i^*_s \omega =0, \,  \lim_{s \to 0} \Pi_{q-1,q}\circ i^*_s d\omega=0\}
\]
is isomorphic to $IH_{(q-1)}^*(W)$.  Furthermore, we have the following long exact sequence on cohomology:
\[
\to H^{j-q-1}(B,H^p(F)) \stackrel{\delta}{\longrightarrow} IH_{(q-1)}^j(W) \stackrel{\inc^*}{\longrightarrow} IH_{(q)}^j(W) 
\stackrel{r}{\longrightarrow} H^{j-q}(B,H^q(F)) \to ,
\]
where $r=\lim_{s \to 0} \Pi_{q-1,q}\circ i^*_s$.
\end{lemma}
These are the complexes used to prove Theorem 5.1 from \cite{Hu3}, so from Corollary \ref{fbmetrics}, letting
$W=CT(X)$, $B=L$, $\Sigma = F$, we have
that the isomorphism in Corollary \ref{cfhodge} is realised by an inclusion of the space of extended weighted 
$L^2$ harmonic forms into the numerator of the quotient space:
\[
\calH^j_{ext}(M',g_{fs},c) \to \frac{{\rm Ker}(d) \subset x^{(f/2)-q - \epsilon}L^2\Omega^j(M,g_{fc})}
{d(x^{(f/2)-q - \epsilon}L^2\Omega_{0}^j(M,g_{fc}))},
\]
where $q=j-\frac{b-1}{2} -c$ for $b$ the dimension of $B$.  We can reinterpret the spaces on the right in terms of the metric $g_{fs}$ to get:
\[
\calH^j_{ext}(M',g_{fs},c) \to \frac{{\rm Ker}(d) \subset x^{c-1 - \epsilon}L^2\Omega^j(M',g_{fs})}
{d(x^{c-1 - \epsilon}L^2\Omega_{0}^j(M',g_{fs}))}.
\]
Using Theorem \ref{ihhi}, we calculate $HI^j_{\dR, \olp}(X)$ from this quotient:
\[
HI^j_{\dR, \olp}(X) \cong \frac{{\rm Ker}(d) \subset x^{(l-3)/2-\olp(l+1) - \epsilon}L^2\Omega^*(M',g_{fs})}
{d(x^{(l-3)/2-\olp(l+1) - \epsilon}L^2\Omega_{0}^*(M',g_{fs}))}.
\]
This is then the definition of $HI^j_{\dR, \olp}(X)$ for which the isomorphism in the Hodge theorem, Theorem \ref{hodge}, is given
by the classical map $\gamma \to [\gamma]$.

\section{Proof of Theorem \ref{intersection}}
In order to prove Theorem \ref{intersection}, we need to understand how the intersection pairing
defined on the original de Rham cohomology of intersection spaces relates to the isomorphism
in Theorem \ref{ihhi} and the intersection pairing on the de Rham intersection cohomology groups.
First, we can show that the alternative complex we defined to calculate $HI^*_{\dR, \olp} (X)$ also admits
a natural intersection pairing by integration, and that this pairing is equivalent to the original
pairing by the isomorphism in Lemma \ref{hieq}.

\begin{lemma}
Integration defines a bilinear pairing between $\widehat{HI}^j_\olp(X)$ and $\widehat{HI}^{n-j}_\olq(X)$
which is equal to the pairing by integration between $HI^j_{\dR, \olp} (X)$ and $HI^{n-j}_{\dR, \olq} (X)$.
\end{lemma}
\begin{proof}
First we will show there is a well defined bilinear pairing between $\widehat{HI}^j_\olp(X)$ and $\widehat{HI}^{n-j}_\olq(X)$. 
Let $\hat{\alpha} \in \widehat{\Omega I}^j_\olp(X)$ and $\hat{\beta} \in \widehat{\Omega I}^{n-j}_\olq(X)$.
Then $\int_M \hat{\alpha} \wedge \hat{\beta}$ is finite since both forms are smooth on $\olM$.
Furthermore, if $\hat{\alpha} = d\hat{\eta}$, where $\hat{\eta} \in\widehat{\Omega I}^{j-1}_\olp(X)$, then
\[
\int_M d\hat{\eta} \wedge \hat{\beta} = \lim_{s \to 0} \int_Y \hat{\eta}(s) \wedge \hat{\beta}(s).
\]
We can decompose $\hat{\eta}(s)$ and $\hat{\beta}(s)$ by $(L,\Sigma)$ bidegree.  By definition
of $\widehat{\Omega I}^{j-1}_\olp(X)$, $\widehat{\Omega I}^{n-j}_\olq(X)$, and by the fact that $\olp(l+1)+\olq(l+1)= l-1$, 
we get that
\[
\lim_{s \to 0} \sum_{i=0}^{k-1} \hat{\eta}_{i,j-1-i}(s)=0= \lim_{s \to 0} \sum_{i=0}^{l-k} \hat{\beta}_{i,(n-j)-i}(s).
\]
Thus the only part that can remain in the limit is
\[
= \lim_{s \to 0} \int_Y \left( \sum_{i=k}^{l} \hat{\eta}_{i,k-1-i}(s)\right) \wedge \left(\sum_{i=l+1-k}^{l} \hat{\beta}_{i,n-j-i}(s)\right).
\]
But this also is zero, since none of the terms in the second sum is of complementary bidegree to any term in the first sum.

 \medskip
Now we can show that this pairing is equal to the pairing by integration between 
$HI^j_{\dR, \olp} (X)$ and $HI^{n-j}_{\dR, \olq} (X)$.  This follows using the surjectivity argument from Lemma \ref{hieq}.  
First note that when we replace the 
lower-truncated multiplicatively structured complex on $Y$ by the lower truncated standard complex on $Y$
(corresponding to the step where we adjust $\beta$ to $\tilde{\beta}$ in Lemma \ref{hieq}), we get the 
same pairing.  This is because multiplicatively structured forms on $Y$ are dense in $L^2$, and in particular,
in the subspace of smooth forms on $Y$.  Thus the $L^2$ pairing on $M$ will extend continuously to lower truncated smooth forms.

So it suffices to consider how the pairing is preserved in the first part of the surjectivity argument, passing from 
$\hat{\alpha}$ to $\alpha$ and $\hat{\beta}$ to $\beta$.  We have that 
\[
\alpha = \hat{\alpha} - d \chi\int_0^x \hat{\alpha}_n(t) \, dt, \qquad \beta = \hat{\beta} - d \chi\int_0^x \hat{\beta}_n(t) \, dt.
\]
Thus
\[
\int_M \alpha \wedge \beta = \int_M \left(\hat{\alpha} - d \chi\int_0^x \hat{\alpha}_n(t) \, dt \right) \wedge 
\left(\hat{\beta} - d \chi\int_0^x \hat{\beta}_n(t) \, dt \right)
\]
\[
= \lim_{s \to 0} \left[ \int_{M_s} \hat{\alpha} \wedge \hat{\beta} - \int_Y\left(\chi(s)\int_0^s \hat{\alpha}_n(t) \, dt \right) \wedge \beta(s)  \pm \int_Y \hat{\alpha}(s)\wedge \left(\chi(s)\int_0^s \hat{\beta}_n(t) \, dt\right) \right]
\]
\[
=\int_{M} \hat{\alpha} \wedge \hat{\beta},
\]
because the other two integrands vanish at $s=0$.
\end{proof}

Next, we want to trace this pairing through the proof of Theorem \ref{ihhi} to see how it can be interpreted
in terms of the signature pairing on intersection cohomology on $CT(X)$.  Recall that we have
\begin{equation}\label{Igp}
HI^j_{\dR,\olp} (X) \cong IG^j_{(t+1)}(CT(X)) 
= \frac{IH^j_{(t)}(CT(X)) \oplus IH^j_{(t+1)}(CT(X))}{{\rm Image}\left(IH^j_{(t)}(CT(X)) \to IH^j_{(t+1)}(CT(X)) \right)}
\end{equation}
where $t=j-l+\olp(l+1)$.  
So also if $\olp$ and $\olq$ are dual perversities on $X$, then $\olp(l+1) + \olq(l+1) = l-1$ implies that
\begin{equation}\label{Igq}
HI^{n-j}_{\dR, \olq} (X) \cong IG^{n-j}_{(s+1)}(CT(X)) 
= \frac{IH^{n-j}_{(s)}(CT(X)) \oplus IH^{n-j}_{(s+1)}(CT(X))}{{\rm Image}\left(IH^{n-j}_{(s)}(CT(X)) \to IH^{n-j}_{(s+1)}(CT(X)) \right)}
\end{equation}
where $s=n-j-1-\olp(l+1)$.  Observe that $t+s+1= n-l$, which is the codimension of the singular stratum
in $CT(X)$.  This is the relationship we expect for the cutoff degrees for dual perversities in $IH^*_*(CT(X))$.
That is, the signature pairing for intersection cohomology on $CT(X)$ pairs the first term in the 
top of Equation (\ref{Igp}) with the second term in the top of Equation (\ref{Igq}), and vice versa.

We can identify the right and left spaces in Equation (\ref{Igp}) in terms of the $HI$ space
using the maps $A$ and $B$ from the proof of Theorem \ref{ihhi}.  To distinguish these
maps in the two settings of Equations (\ref{Igp}) and (\ref{Igq}), fix the following
notation:
\[
IH^j_{(t)}(CT(X))\stackrel{A_\olp}{\longrightarrow} HI^j_{\dR,\olp} (X) \stackrel{B_\olp}{\longrightarrow} IH^j_{(t+1)}(CT(X))
\]
and similarly define $A_\olq$ and $B_\olq$ for the spaces in Equation (\ref{Igq}).  Then we have
\[
IH^j_{(t)}(CT(X)) \cong \mbox{Im}(A_\olp), \mbox{ and} \]
\[
IH^j_{(t+1)}(CT(X)) \cong HI^j_{\dR,\olp} (X)/\mbox{Ker}(B_\olp),
\]
and analogous isomorphism in the $\olq$ case.  Now we can precisely state the compatibility
between the intersection pairing on $HI$ spaces and on $IH$ spaces.
\begin{lemma}\label{comp}
For $[\alpha] \in IH^j_{(t)}(CT(X))$ and $[\beta] \in HI^{n-j}_{\dR,\olq} (X)$, 
\[
A_\olp[\alpha] \cap_{HI} [\beta] = [\alpha] \cap_{IH} B_\olq[\beta].
\]
\end{lemma}
\begin{proof}
Both of the pairings, $\cap_{IH}$ and $\cap_{HI}$ are achieved on their corresponding de Rham
cohomology spaces by integration of the wedge of representatives of the paired cohomology classes.
Both are known to be well-defined on their corresponding cohomologies.  Furthermore, by definition
of the map $A_\olp$, we can take the same
representative form to represent both $[\alpha]$ and $A_\olp[\alpha]$.  Similarly, we can represent
both $[\beta]$ and $B_\olq[\beta]$ by the same form.  
Thus for $[\alpha]$ and $[\beta]$ as in the statement of the lemma,
\[
A_\olp[\alpha] \cap_{HI} [\beta]:= \int_M \alpha \wedge \beta := [\alpha] \cap_{IH} B_\olq[\beta].
\]
\end{proof}
Note that this gives us the following corollary:
\begin{cor}
$\mbox{Image}(A_\olp)$ is the annihilator under the pairing $\cap_{HI}$ of $\mbox{Kernel}(B_\olq)$.
\end{cor}
\begin{proof}
If $[\beta] \in \mbox{Kernel}(B_\olq)$, then by Lemma \ref{comp}, $A_\olp[\alpha] \cap_{HI} [\beta]=0$.  This 
means that $\mbox{Image}(A) \subset \mbox{Ann}(\mbox{Kernel}(B_\olq))$.  Note that since 
$B \circ A$ is the natural map
of adjacent intersection cohomology groups obtained by the inclusion of cochain complexes, we have
\[
\mbox{Kernel}(B_\olq) \cong \mbox{Kernel}(B_\olq\circ A_\olq), \qquad \mbox{and}
\]
\[
HI^j_{\dR,\olp}(X)/\mbox{Image}(A_\olp) \cong IH^j_{(t+1)}(CT(X))/\mbox{Image}(B_\olp \circ A_\olp).
\]
Thus by Poincar\'e duality on intersection cohomology, 
\[
\mbox{dim}\left(\mbox{Kernel}(B_\olq)\right) = \mbox{dim}\left(HI^j_{\dR,\olp} (X)/\mbox{Image}(A_\olp)\right).
\]
So by nondegeneracy of the intersection pairing, in fact $\mbox{Image}(A_\olp)$ is the entire annihilator
of $\mbox{Kernel}(B_\olq)$.
\end{proof}

Now let us focus on the setting where $X$ is even dimensional and has a unique middle perversity,
$\olm$.  Then the equations \ref{Igp} and \ref{Igq} are identical, and we can check that
\[
IH^{n/2}_{(s)} (CT(X))=IH^{n/2}_{\olm}(CT(X)), \qquad IH^{n/2}_{(s+1)} (CT(X))=IH^{n/2}_{\oln}(CT(X)),
\]
the middle degree lower and upper middle perversities for $(CT(X))$.  Now we can use the 
nondegeneracy of the intersection pairing on $HI^{n/2}_{\dR,\olm}(X)$ to identify the dual
of $\mbox{Kernel}(B_\olm)$ as a subspace of $HI^{n/2}_{\dR,\olm} (X)$.  For the sum of 
$\mbox{Kernel}(B_\olm)$ and its dual, the signature form then vanishes.  This means that 
the signature on $HI^{n/2}_{\dR,\olm} (X)$ is equal to the signature on the complement of this space,
which is isomorphic to 
\[
\mbox{Image}(A_\olm)/\mbox{Kernel}(B_\olm) \cong \mbox{Image}(B_\olm \circ A_\olm) \subset IH^{n/2}_\oln(CT(X)).
\]  
By Lemma \ref{comp}, the signature on two sides of this isomorphism are also equal.  Thus
we get that the signature of the intersection pairing on  $HI^{n/2}_{\dR,\olm} (X)$ is equal to the 
signature of the intersection pairing on 
\[
\mbox{Image}(IH^{n/2}_\olm(CT(X)) \to IH^{n/2}_\oln(CT(X)),
\]
which is by definition the middle perversity perverse signature on $CT(X)$.

Next, we prove that both of these are also equal to the middle
perversity signatures for $HI$ and $IH$ of the space $Z$ obtained as the one-point compactification of $M$
(which are both simply the signature of $M$ as an open manifold).  This follows from
a result in \cite{Hu2}, which calculates perverse ($IH$) signatures for a pseudomanifold with a single
smooth singular stratum as the sum of the signature on its complement (i.e., the signature of $M$)
and a set of terms arising from the second and higher pages in the Leray spectral sequence of 
the link bundle of the singular stratum.  In particular, if the spectral sequence degenerates at the 
second page, as it does in the case of a product bundle, all of these additional terms vanish, 
so all perverse signatures are simply the signature of $M$.

It remains to show that $\sigma_{IH} (X) = \sigma_{IH} (Z)$.
There are several ways to see this, for example as follows:
By Siegel's pinch bordism (cf. \cite{siegel} or \cite[Chapter 6.6]{banagl-tiss}),
$\sigma_{IH} (X)= \sigma_{IH} (Z) + \sigma_{IH} (E),$ where $E$ is the pseudomanifold
\[ E = (cL)\times \Sigma \cup_{L\times \Sigma} c(L\times \Sigma). \]
If $l=\dim L$ is odd, then Lemma 8.1 of \cite{bcs} implies that in fact already
the group $IH^{\olm}_{n/2} (E)$ is trivial. In particular, $\sigma_{IH} (E)=0$ and
$\sigma_{IH} (X) = \sigma_{IH} (Z)$.
If $l$ is even, then $\dim \Sigma$ is odd and thus $CT(X)$ is a Witt space.
(Note that $\dim \Sigma$ odd means in particular that $\dim \Sigma \geq 1$ and thus that
the singular set of $CT(X)$ has codimension at least $2$.)
Hence we may apply what we have proved so far to $X' = CT(X)$ and obtain
\[
\sigma_{HI}(CT(X)) = \sigma_{IH,\olm}(CT(CT(X))) 
=\sigma_{IH}(Z). 
\] 
Since $CT(CT(X))\cong X$ and $X$ is Witt, we have for the perverse signature
\[ \sigma_{IH,\olm}(CT(CT(X))) =\sigma_{IH} (X). \]

\section*{Acknowledgements}
The authors thank the Deutsche Forschungsgemeinschaft for funding the research visits during 
which much of this work was done. The second author thanks Timo Essig and Bryce Chriestenson 
for useful discussions.  In particular, the proof of Lemma \ref{HIkun} is partly due to Essig 
and the proof of Lemma \ref{HIMV} is due to Chriestenson.

\end{document}